\documentclass[11pt]{amsart}
\usepackage[T1]{fontenc}
\usepackage[utf8]{inputenc}
\usepackage[english]{babel}
\usepackage{csquotes} 
\usepackage[a4paper,margin=3cm]{geometry}
\usepackage{graphicx}
\usepackage[backend=biber,style=numeric,bibencoding=utf8,language=auto,autolang=other,giveninits=true,doi=false,isbn=false,url=false,maxnames=10,sorting=nty]{biblatex}
\usepackage{caption} 
\usepackage[hyperfootnotes=false]{hyperref} 
\hypersetup{
	colorlinks = true,
	linkcolor = {blue},
	urlcolor = {red},
	citecolor = {blue}
}
\usepackage{footmisc}
\usepackage[nameinlink,capitalise,sort,noabbrev]{cleveref}
\crefname{equation}{}{} 
\crefname{enumi}{}{} 

\usepackage{amsmath}
\usepackage{amsthm}
\usepackage{amssymb}
\usepackage{esint} 
\usepackage{mathrsfs} 
\usepackage{faktor} 
\usepackage{dsfont} 
\usepackage[dvipsnames]{xcolor}
\usepackage{array}
\usepackage{hhline}
\usepackage{enumitem} 
\usepackage{comment} 
\usepackage[toc,page]{appendix} 
\usepackage{imakeidx} 
\usepackage{xparse} 
\usepackage{mathtools}
\usepackage{fancyhdr} 
\usepackage{ifthen} 
\usepackage{forloop} 
\usepackage{xstring}
\usepackage{tikz}
\usetikzlibrary{babel} 
\usetikzlibrary{cd} 
\usepackage{emptypage} 
\usepackage{listings} 
\usepackage{mathabx} 

\theoremstyle{plain}
\newtheorem{lemma}{Lemma}[section]

\newtheorem{theorem}[lemma]{Theorem}
\newtheorem{corollary}[lemma]{Corollary}

\newtheorem{problem}{Problem}
\newtheorem{thmx}{Theorem}

\theoremstyle{definition}

\theoremstyle{remark}
\newtheorem{remark}[lemma]{Remark}
\newtheorem{example}[lemma]{Example}

\numberwithin{equation}{section}

\newcommand{\pt}{\partial_t}

\newcommand{\rS}{\mathrm{S}}

\renewcommand{\d}{\mathrm{d}}

\DeclareMathOperator{\supp}{supp}

\newcommand{\loc}{\mathrm{loc}}
\newcommand{\R}{\mathbb{R}}
\newcommand{\N}{\mathbb{N}}
\newcommand{\T}{\mathrm{T}}
\newcommand{\eps}{{\varepsilon}}

\newcommand{\bigo}{\ensuremath{\mathcal{O}}}

\newcommand{\Leb}{\ensuremath{\mathscr L}} 

\newcommand{\ie}{{\itshape i.e.}}
\newcommand{\eg}{{\itshape e.g.}} 
\newcommand{\cf}{{\itshape cf.}~}

\newcommand\Restr[1]{%
   \raisebox{-0.35ex}{$\scalebox{1.05}[1.15]{$\vert$}%
                      _{\scriptscriptstyle #1\mathstrut}$}}

\begin{document}

\title[Optimal transport via autonomous vector fields]{Optimal transport of  measures via autonomous vector fields}



\author[N.~De Nitti]{Nicola De Nitti}
\address[N.~De Nitti]{EPFL, Institut de Mathématiques, Station 8, 1015 Lausanne, Switzerland.}
\email{nicola.denitti@epfl.ch}

\author[X.~Fern\'andez-Real]{Xavier Fern\'andez-Real}
\address[X.~Fern\'andez-Real]{EPFL, Institut de Mathématiques, Station 8, 1015 Lausanne, Switzerland.}
\email{xavier.fernandez-real@epfl.ch}


\begin{abstract}
We study the problem of transporting one probability measure to another via an autonomous velocity field. We rely on tools from the theory of optimal transport. In one space-dimension, we solve a linear homogeneous functional equation to construct a suitable autonomous vector field that realizes the (unique) monotone transport map as the time-$1$ map of its flow. Generically, this vector field can be chosen to be Lipschitz continuous. We then use Sudakov's disintegration approach to deal with the multi-dimensional case by reducing it to a family of one-dimensional problems.

\end{abstract}

\maketitle

\section{Introduction}
\label{sec:intro}

\subsection{The problems} We are interested in the problem of transporting one probability measure into another using an autonomous vector field.

This problem can be viewed from two perspectives. The \emph{Lagrangian} one involves pushing the  first measure forward to the second via the time-$1$ map of the flow generated by the vector field:

\begin{problem}[Matching  measures via the flow generated by an autonomous vector field] \label{prob:1}
Given two probability measures $\mu_0, \mu_1 \in \mathcal P(\R^d)$, with $d\ge 1$, construct an autonomous vector field $v:\mathbb{R}^d \to \mathbb{R}^d$ for which there is a flow\footnote{~The ODE \cref{eq:ode} is interpreted in the following sense:  $t \mapsto \phi(t,x)$ is absolutely continuous and 
\[
\phi(t, x)=x+\int_0^t v(\phi(s,x)) \, \d s, \quad \text{ for all } t \ge 0,
\] holds for all $x \in \R^d$.},  
\begin{align}\label{eq:ode}
    \begin{cases}
    \pt \phi(t,x) = v(\phi(t,x)), & t >0, \ x \in\R^d \\
    \phi(0,x) = x, & x \in \R^d,
    \end{cases}
\end{align}
that satisfies
\begin{align}\label{claim:exact}
    \phi(1,\cdot)_{\#}\mu_0 \equiv \mu_1.
\end{align}
\end{problem}
We recall that the measure denoted by $\phi(1,\cdot)_{\#} \mu_0$ is defined by  \[\left(\phi(1,\cdot)_{\#} \mu_0\right)(A)\coloneqq \mu_0\left(\phi(1,\cdot)^{-1}(A)\right), \quad \text{for every measurable set $A \subset \R^d$},
\]
and is called \emph{image measure} or \emph{push-forward} of $\mu_0$ through $\phi(1,\cdot)$.

\cref{prob:1} amounts to a question about \emph{exact controllability} for ordinary differential equations.

The second perspective is \emph{Eulerian} and involves a question of {exact controllability} for the continuity (partial differential) equation: steering the solution of the continuity equation from an initial state $\mu_0$ to a target state $\mu_1$ by using an autonomous velocity field $v$ as a control:

\begin{problem}[Exact controllability of the continuity  equation using an autonomous velocity] \label{prob:1bis}
Given two probability measures $\mu_0, \mu_1 \in \mathcal P(\R^d)$, with $d\ge 1$,  construct an autonomous vector field $v:\mathbb{R}^d \to \mathbb{R}^d$ such that a solution $\mu:[0,+\infty) \times \R^d \to \R$ to the Cauchy problem
\begin{align}\label{eq:tr}
    \begin{cases}
    \pt \mu(t,x) + \operatorname{div}_x(v(x) \,  \mu(t,x))= 0, & t >0, \ x \in \R^d, \\
    \mu(0,x) = \mu_0(x), & x \in \R^d,
    \end{cases}
\end{align}
 satisfies
\begin{align}\label{claim:exact-bis}
    \mu(1,\cdot) \equiv \mu_1.
\end{align}
\end{problem}

If  $v$ is smooth  then, by the \emph{method of characteristics}, the (unique) solution $\mu$ of \cref{eq:tr} can be represented using the (unique) flow $\phi$ of $v$, and vice-versa. 

More generally, if both $\mu$ and $\phi$ exist and are unique, the equivalence between \cref{prob:1} and \cref{prob:1bis} is a consequence of the \emph{Lagrangian representation formula} for the solution of \cref{eq:tr}: 
\begin{align}\label{eq:lagrangian-formula}
\mu(t,\cdot) \equiv \phi(t,\cdot)_\# \mu_0, \qquad t \ge 0.
\end{align}
This is not generally the case if we drop the uniqueness assumption on $\mu$. For example, even when \cref{eq:ode} has a unique flow and \cref{eq:lagrangian-formula} represents a solution (called the \emph{Lagrangian solution}) of \cref{eq:tr}, it does not necessarily encompass all solutions\footnote{\label{fn:lagrangian} ~In one space-dimension, if the vector field is continuous and autonomous, owing to  \cite[Proposition 5]{zbMATH07552043}, uniqueness for the  \cref{eq:ode} implies that every solution of \cref{eq:tr} is represented by \cref{eq:lagrangian-formula} and, in particular, uniqueness for \cref{eq:tr}. When we drop the continuity assumption, this is generally false (see \cite{MR4002209}). 

In any space-dimension, the claim is true for \emph{non-negative measures} by Ambrosio's superposition principle (see \cite[Theorem 8.2.1]{MR2401600}); however, the superposition principle cannot be extended to signed solutions (see \cite{MR3778560}). On the other hand, in the class of signed measures, the claim still holds true, \eg, if the velocity field is either Lipschitz continuous (see \cite[Proposition 8.1.7]{MR2401600}), or log-Lipschitz continuous (see \cite[Théorème 5.1]{MR1288809}), or satisfying a quantitative two-sided diagonal Osgood condition (see \cite[Theorem 1]{MR2439520}). 
}; therefore, solving \cref{prob:1} provides a solution to \cref{prob:1bis}, but the converse does not necessarily hold. We refer to 
\cite{MR4002209,zbMATH07552043} for a discussion on the validity of \cref{eq:lagrangian-formula}.

The construction of an autonomous vector field addressing \cref{prob:1} (or \cref{prob:1bis}) is not difficult to do if $\mu_0$ and $\mu_1$ are superpositions of Dirac deltas. Conversely, when the measures are not just superposition of deltas, even for $d = 1$ such a construction becomes more delicate. In this work, we focus our attention on solving \cref{prob:1} and \cref{prob:1bis} in the case when $\mu_0$ is absolutely continuous with respect to the Lebesgue measure, \ie, $\mu_0 \ll \mathscr L^d$, and has a continuous density. Our strategy is based on tools from the theory of optimal transport of measures (see, \eg, \cite{MR2459454,MR1964483,MR4659653,MR4294651,MR4655923,MR3409718,MR3058744} for an overview of the topic). In particular, we build an autonomous velocity field from a given Monge's optimal transport map in dimension $d = 1$. That is, we turn to the following question:

\begin{problem}[Realizing a given optimal transport map as time-$1$ map of the flow associated with an autonomous vector field]\label{prob:2}
Given two probability measures $\mu_0, \mu_1 \in \mathcal P_{\mathrm{a.c.}}(\R)$, with $d = 1$,   construct an  autonomous vector field $v: \R \to \R$  such that 
    \[\phi(1,\cdot) \equiv \T, \]
where $\phi$ solves \cref{prob:1}, and  $\T$ is the unique monotone map with $\T_\#\mu_0 = \mu_1$.
\end{problem}

\subsection{Our results and structure of the paper} First, in \cref{sec:1d}, we solve \cref{prob:2} (and, as a byproduct, \cref{prob:1} and \cref{prob:1bis} in case $d = 1$). In one space-dimension, the results of the theory of optimal transport are very sharp: provided that the source measure $\mu_0$ has no atoms, there exists only one monotone non-decreasing transport map (which is optimal for the cost $\mathrm{c}(x,y) \coloneqq |x-y|^p$, with $p\ge 1$), as recalled in \cref{th:ot1d}. We show that this map can be realized as time-$1$ map of an autonomous vector field.

\begin{thmx}[Exact controllability, $d=1$]
\label{th:A}

Let   $\mu_0, \mu_1 \in \mathcal{P}_{\mathrm{a.c.}}(\R)$ be two probability measures with convex support\footnote{~We recall that, given a (non-negative) measure $\mu$ on a measurable topological space $(X,\Sigma = {\rm Borel}(X))$, $\supp \mu \coloneqq \overline{\{A\in \Sigma:\, \mu(A) \neq 0\}}$. In particular, the support of a measure is a closed set and, for any compact $K\subset \R$, under our assumptions, we have that the densities are bounded from below by a positive constant in $K\cap {\rm supp}(\mu_i)$.}, and continuous densities positive in their supports. Then, there exists a solution to \cref{prob:1} and \cref{prob:2}. 

More precisely, there exists an autonomous velocity field $v$ and a unique solution  $\phi$ to \cref{eq:ode}, up to time $t = 1$, which  satisfies $\T \equiv \phi(1, \cdot)$ in ${\supp \mu_0}$, and thus \cref{claim:exact}, where $\T$ is the unique monotone transport map between $\mu_0$ and $\mu_1$.  
\end{thmx}

Some remarks are in order: 

\begin{remark}[Non-uniqueness of the velocity field]\label{rk:non-uniqueness1}
    The velocity field is non-unique and, in general, as we will see in its construction in \cref{ssec:proof-1d}, it is obtained from an arbitrary prescription in an open set. More precisely, given $x_0\in{\rm supp}(\mu_0)$, we \emph{arbitrarily} fix the velocity field in $(x_0, \T(x_0))$ and then extend it \emph{uniquely} to the interval between consecutive fixed points of $\T$ containing $x_0$. Morally, this is enough because the final position of any $y\in (x_0, \T(x_0))$ can be modulated only by the values of the velocity field in $(\T(x_0), \T(y))$. 
\end{remark}

\begin{remark}[On the solution to \cref{prob:1bis}]
\label{rmk:prob2}
    The $v$ constructed also gives a solution to \cref{prob:1bis} in a suitable sense. That is, the function $\mu(t, \cdot) = \phi(t, \cdot)_\# \mu_0$ satisfies \cref{eq:tr} as follows: 
    There exists a discrete set $\partial\mathcal{S} = \partial\{x = \T(x)\}$ where $v = 0$ such that $\mu(t,\cdot)$ satisfies \cref{eq:tr} in the distributional sense in ${\rm supp}(\mu(t,\cdot))\setminus \partial\mathcal{S}$ and also a no-flow condition through $\partial \mathcal{S}$ (namely,  trajectories starting outside of $\partial\mathcal{S}$ never reach $\partial \mathcal{S}$ in finite time). 
   
    The need for the previous notion is because, as we will show in \cref{lm:ct}, the velocity fields constructed do not have to be $L^1_{\rm loc}$ in general at points $\partial\mathcal{S}$, and thus the $\mu(t,\cdot)$ above need not be a distributional solution across $\partial\mathcal{S}$.          Somewhat related notions of solutions have been employed in  \cite{zbMATH03737992,zbMATH06243849,zbMATH07599881} (in different contexts\footnote{~In particular, in the autonomous setting, Aizenman, in \cite{zbMATH03737992}, proved that a suitable generalized flow avoids a subset $S\subset \R^d$ provided that the vector field is sufficiently regular and $S$ has sufficiently small box-counting dimension. In \cite{zbMATH06243849}, this result was extended to the non-autonomous setting. More recently, in \cite{zbMATH07599881}, the authors proved  Ambrosio's uniqueness result (see \cite{MR2096794}) by allowing the presence of a compact set of singularities $S \subset [0,T]\times \R^d$, such that $b|_{\Omega} \in L^1([0,T]; \mathrm{BV}_{\loc}(\R^d))$ for all compact sets $\Omega \subset S^c$.}).  

    Finally, we note that such solutions are the unique distributional solutions on ${\rm supp}(\mu(t,\cdot))\setminus \partial\mathcal{S}$.  Indeed, this follows from the uniqueness of the flow for continuous and signed velocities in  \cite[Proposition 5.2]{MR4002209}, which we can apply in the open intervals between fixed points.

\end{remark}

\begin{remark}[On the positivity assumption]
As it will be clear from the proofs, the assumption that the densities of $\mu_i$ are positive in ${\rm supp}(\mu_i)$ could be weakened to being positive only in ${\rm int}({\rm supp}(\mu_i))$ instead. In this case, however, velocities would blow-up or vanish at the endpoints of the supports. On the other hand, removing the positivity assumption in the interior would allow for velocities blowing up or vanishing in the interior as well, interfering with the notions of solution used to make sense of the previous problems.
\end{remark}

  We refer to \cref{th:main-1d} for the precise statement of \cref{th:A}.  In particular, under suitable assumptions, we obtain further structure and regularity properties for the velocity field. The proof of this result is connected to the theory of linear homogeneous functional equations (see  \cite{MR1994638, zbMATH03252776, zbMATH03164897,zbMATH03445777,zbMATH03312170, zbMATH00194090,   zbMATH03618689, MR557216}).

We would like to remark that, in the previous construction, the arising vector fields $v$ are not necessarily continuous in general (nor it is expected, even though, in \cref{rmk:preva}, we show that they are generically Lipschitz continuous), and they are only piecewise continuous.  Even with that, we show the well-posedness (in particular, uniqueness) of the flow. On the other hand, as we show in \cref{lm:ct}, in some settings, the only candidates for $v$ that are admissible either do not generate unique flows or are not even $L^1$. If we wanted a well-posedness theory both for the Lagrangian and Eulerian formulations, we would need to weaken some of the assumptions. This is the goal of the following result: if we relax the requirement that $\mu_1$ should be achieved \emph{exactly}, it is possible to further improve the (global) regularity of the velocity field $v$ to make it Lipschitz (see \cref{prop:approx} for a more precise statement), and thus be able to use the full strength of the well-posedness theory for both the flow and the continuity equation:

\begin{thmx}[Approximate controllability, $d=1$]
\label{th:B}
In the setting of \cref{th:A}, if, furthermore, $\mu_0$ has a Lipschitz continuous density, then for any $\eps > 0$, there exists $\mu_1^{\eps}$ satisfying the same hypotheses as $\mu_1$ and with\footnote{~Here $\mathrm{dist}(\mu_1, \mu_1^{\eps}) < \eps$ can be understood either  in the $L^1$ or in the Wasserstein sense.} $\mathrm{dist}(\mu_1, \mu_1^{\eps}) < \eps$ such that the corresponding vector field $v_\eps$ from \cref{th:A} can be taken Lipschitz continuous.  
\end{thmx}

 In particular, we also have a unique solution to \cref{prob:1bis} with vector field $v_\eps$ transporting $\mu_0$ into $\mu_1^\eps$ (\cf  \cref{rmk:prob2}).

Building on these results, in \cref{sec:multid}, we deal with the case $d \ge 2$. We use \emph{Sudakov's disintegration approach} (see \cite[Chapter 18]{MR4659653}) to decompose the multi-dimensional optimal transport problem into a family of one-dimensional problems, namely, optimal transport problems on a family of \emph{optimal transport rays} that forms a partition of ${\mathrm{Conv}(\supp \mu_0 \cup \supp \mu_1)}$.

Sudakov's optimal transport map can be written as the ``gluing'' the one-dimensional monotone optimal transport maps built along the transport rays. Correspondingly, we are able to build a vector field in $\R^d$ by putting together the one-dimensional vector fields constructed previously (see  \cref{th:main-d} for a precise statement).

\begin{thmx}[Exact controllability, $d \ge 1$]\label{th:C}
Let $\mu_0, \mu_1 \in \mathcal{P}_{\mathrm{a.c.}}(\R^d)$, with $d\ge 1$,  be two probability measures with convex support and continuous densities positive in their supports. Then there exists a solution to \cref{prob:1}.

More precisely, there exists an autonomous velocity field $v$ transporting $\mu_0$ into $\mu_1$ in the sense of \cref{claim:exact-bis} and there exists a  solution  $\phi$ to \cref{eq:ode}, up to time $t = 1$, which  satisfies $\T \equiv \phi(1, \cdot)$ in ${\supp \mu_0}$, and thus   \cref{claim:exact}, where $\T$ is Sudakov's optimal transport map between $\mu_0$ and $\mu_1$.
\end{thmx}

We note that, for \cref{th:C}, we no longer have the uniqueness of the flow; see \cref{rmk:uniqueness} for a short discussion. 

Finally, in \cref{sec:examples}, we illustrate these results by presenting some one-dimensional examples.

\subsection{Some related results}
\label{ssec:lit-rev}

In the available literature, \cref{prob:1} and \cref{prob:1bis} have been extensively analyzed in case the requirement on $v$ being autonomous is dropped. 

For example, the pioneering construction performed by Dacorogna and Moser  in \cite{MR1046081} provides a \emph{time-dependent} velocity field realizing \cref{claim:exact}: 
\[
v(t,x)\coloneqq \frac{ \nabla f(x)}{\bar \mu_0(1-t)+\bar \mu_1 t},
\]
where $\bar \mu_i$ denotes the (smooth and positive) density of $\mu_i$ (for $i \in \{0,1\}$), $f \in C^{\infty}(\R^d)$ is the unique solution of $-\Delta f=\bar \mu_1-\bar \mu_0$ with zero mean.

More recently, in \cite{MR3936028,MR4081517}, Duprez, Morancey, and Rossi constructed a \emph{time-dependent} and localized perturbation of a given velocity field to achieve \cref{claim:exact-bis} (the localization being in a given non-empty, open, and connected portion of $\R^d$). 

Finally, in \cite{MR4624336,ASZ24,zbMATH07786288}, \cref{prob:1} and \cref{prob:1bis} were studied with ``neural'' velocity functions, \ie, under the ansatz $v(t,x) \coloneqq  w(t)\sigma (\langle a(t), x\rangle + b(t))$, with $\sigma (x) \coloneqq \max\{ x, 0\}$ (the so-called \emph{activation function of the neural network}) and control parameters $a, w \in L^\infty ((0, 1); \R^d)$ and
$b \in L^\infty ((0, 1); \R)$. The controls $a$, $w$, and $b$ were constructed \emph{piecewise-constant in time} (with an explicit bound on the number of jumps). More recently, a class of so-called ``semi-autonomous'' neural velocities have also been considered in \cite{li2024}.

The problem of identifying if a given map can be ``embedded in a flow'' has a long history in the dynamical system community (see, \eg,  \cite{zbMATH07298491,zbMATH03113311} and references therein).  Moreover, the study of inverse problems for some ODEs (namely, reconstructing a vector field from the time-$t_i$ map of the flow for some $\{t_i\}_{i \in \{1, \dots, N\}}$) has been considered in \cite{zbMATH07239837} (and references therein). More recently, the same question was addressed in \cite[Sections 3.1 \& 3.2]{2308.01213} under the additional ansatz that $v$ is of neural type.

\section{Construction in the one-dimensional case}
\label{sec:1d}

If $d=1$,   \cref{eq:ode} reduces to 
\begin{align}\label{eq:ode-1d}
    \begin{cases}
    \pt \phi(t,x) = v(\phi(t,x)), & t >0, \ x \in \R, \\
    \phi(0,x) = x, & x \in \R.
    \end{cases}
\end{align}
We start by remarking that, if the flow is unique (and defined up to time $t=1$), then the map $\R \ni x \mapsto \phi(1,x)$ is  non-decreasing (see \cite[Section 8, XI. Theorem, p. 69]{MR271508})\footnote{~This monotonicity statement is true also for non-autonomous velocities. Let $\phi$ be the unique solution to
\begin{align}\label{eq:cauchy-mono}
\begin{cases}
\partial_t \phi = V(t,\phi(t,x)), & t >0, \\
\phi(0,x) = x, & x \in \R,
\end{cases}
\end{align}
where $V:\R_+\times \R \to \R$. We claim that $x \mapsto \phi(1,x)$ is non-decreasing. 

Let us suppose, for the sake of finding a contradiction, that there exists $x_1 \le x_2$ such that $\phi(1,x_2)<\phi(1,x_1)$. Since $t \mapsto \phi(t,\cdot)$ is a continuous function, we can apply the intermediate-value theorem: 
$x_2 = \phi(0,x_2)>\phi(0,x_1)=x_1$ and $\phi(1,x_2) < \phi(1,x_1)$ imply that $\phi(\bar t, x_2) = \phi(\bar t, x_1) \eqqcolon \bar \phi$ for some $\bar t \in (0,1)$.
This means that $\phi(t,x_1)$ and $\phi(t,x_2)$ solve the Cauchy problem 
\begin{align}\label{eq:cauchy-2}
\begin{cases}
\partial_t \psi(t) = V(t,\psi(t))), & t >\bar t, \\
\psi(\bar t) = \bar \phi, & x \in \R. 
\end{cases}
\end{align} 
This yields a contradiction because the solution of \cref{eq:cauchy-mono} is unique. }. 
Therefore, if a velocity $v:\R \to \R$ exists such that the corresponding flow $\phi$ is unique and satisfies $\phi(1, \cdot)_\#\mu_0 = \mu_1$, then $\phi(1,\cdot)$ must coincide with the unique \emph{monotone transport map} between $\mu_0$ and $\mu_1$, which is optimal for \emph{Monge's optimal transport problem} 
\[
\mathrm{M}(\mu_0, \mu_1) \coloneqq     \min \left\{\int_{\R} \mathrm{c}(\T (x),x) \, \mathrm d \mu_0(x): \ \T : \mathbb{R} \rightarrow \mathbb{R} \text{ and } \mu_1=\T _{\#} \mu_0\right\},
\]
with cost $\mathrm{c}(x, y)\coloneqq |x-y|^p$ for some $p \geq 1$. Owing to the uniqueness of the monotone transport map, \cref{prob:1} and \cref{prob:2}  coincide in this setting.

For the sake of completeness, we recall some known properties of the one-dimensional optimal transport map in the following theorem (see~\cite[Theorem 3.1]{MR2011032} and \cite[Lemma 2.5]{zbMATH02228669}).

\begin{theorem}[One-dimensional Monge's problem]\label{th:ot1d}
Let   $\mu_0, \mu_1 \in \mathcal P(\R)$ and let us assume that $\mu_0$ is \emph{non-atomic} (\ie, a \emph{diffuse measure}: $\mu_0(\{x\})=0$ for any $x \in \R$). Then there exists a unique (modulo countable sets) non-decreasing function $\T  :{\supp  \mu_0} \to \R$ such that $\T _{\#} \mu_0 \equiv \mu_1$, given explicitly by 
\[
\T (x)= \sup \left\{z \in \mathbb R: \, \mu_1((-\infty, z]) \leq \mu_0((-\infty, x])\right\},\quad\text{for}\quad x \in {\supp \mu_0}.
\]
Moreover, the function $\T $ is an optimal transport map (the unique optimal
transport map if $p >1$) and, provided that $\supp  \mu_1$ is connected, it is continuous. Finally, if $\mu_0, \mu_1 \ll \Leb^1$ and their densities $\bar \mu_0$ and $\bar \mu_1$ are continuous and positive functions in their respective supports, then 
$\T $ is $C^1$ and its derivative is given by 
\begin{align}\label{eq:derivative-t}
    \T '(x) = \frac{\bar \mu_0(x)}{\bar \mu_1(\T (x))}.
\end{align}
\end{theorem}

The main result of this section solves, in particular, \cref{prob:2} and, equivalently, \cref{prob:1}, for $d = 1$, and is the following: 

\begin{theorem}[Exact controllability, $d=1$]\label{th:main-1d}
    Let us consider $\mu_0, \mu_1 \in \mathcal{P}_{\mathrm{a.c.}}(\R)$ and suppose that their densities, $\bar \mu_0$ and $\bar \mu_1$, are continuous functions (in their respective supports). Let us suppose that the following conditions hold: 
    \begin{enumerate}[label=\textbf{M-\arabic*}] 
        \item\label{it:m1} $\supp \mu_0$ and $\supp \mu_1$ are convex; 
        \item\label{it:m3} $\bar\mu_0 > 0$ in ${\rm supp} (\mu_0)$ and $\bar \mu_1 > 0$ in ${\rm supp}(\mu_1)$; 
    \end{enumerate} 
    Then there exists a  velocity field $v: {{\rm Conv}(\supp \mu_0 \cup \supp \mu_1)} \to \R$  such that 
    \[
    \begin{aligned}
    |v|> 0  &\quad\text{in}  \quad {{\rm Conv}(\supp \mu_0 \cup \supp \mu_1)}\,\setminus \mathcal{S}, \\ 
    v \equiv 0 &\quad\text{in}\quad {\mathcal{S}},
   \end{aligned}
    \]
and 
\[
    \T (x) = \phi(1,x), \qquad x \in {\supp \mu_0},
    \]
    where  $\mathrm T$ is the monotone optimal transport map from \cref{th:ot1d},  $\mathcal S$ is the set of fixed points of the map $\T$ in ${\supp \mu_0}$, and $\phi$ is the unique solution of \cref{eq:ode-1d} for $x \in { \supp \mu_0 }$.
    
    Moreover, $v$ is continuous except possibly at $\partial \mathcal{S}$. If, additionally, $|\bar\mu_0- \bar \mu_1|> 0$ in $\partial\mathcal{S}$, then $v$ can be taken to be continuous also at $\partial \mathcal{S}$. 
    If,  furthermore, $\bar \mu_0$ and $\bar \mu_1$ are Lipschitz continuous, $v$ can be taken locally Lipschitz continuous up to $\partial \mathcal{S}$.
\end{theorem}

\begin{remark}[Continuity equation]
In \cref{th:main-1d}, we claim the uniqueness of $\phi$. On the other hand, for the continuity equation \cref{eq:tr}, we cannot make any claim, as the velocity field will not, in general, be $L^1_{\mathrm{loc}}$ (and therefore, things would depend on the notion of solution chosen in such a case; \cf \cref{rmk:prob2}); see \cref{lm:ct} below.  

If we are in a situation where $v$ is continuous, we can define weak solutions for \cref{eq:tr} and their uniqueness follows from the uniqueness for the ODE, \cref{eq:ode-1d}, see \cite[Proposition 5.2]{MR4002209}. 

If $v$ is also Lipschitz continuous, then uniqueness for \cref{eq:ode-1d} follows from the Cauchy--Lipschitz theorem and uniqueness for \cref{eq:tr} follows from the classical methods of characteristics. 

\end{remark}

\begin{remark}[Higher regularity]
  If, in addition, $\bar\mu_0\in C^k({\supp\mu_0})$ and $\bar \mu_1\in C^k({\supp\mu_1})$ for some $k\in \N$, then we can choose $v\in C^k({\mathrm{Conv}(\supp \mu_0 \cup \supp \mu_1)}\setminus \partial S)$.   This is a consequence of  \cref{cor:Ck}, stated below. 
\end{remark}

\begin{remark}[Support of $\mu_0$ with two connected components]\label{rmk:supp2}
    If $\supp \mu_0 \equiv C_1 \cup C_2$ with $C_1$ and $C_2$ disjoint connected sets satisfying
    $\sup C_1 < \inf C_2$, then we can still solve the problem by splitting it in two. We consider $\mu^1_0 \coloneqq \mu_0 \Restr{C_1}$, $\mu_0^2 \coloneqq \mu_0\Restr{C_2}$ and $\mu_1^1 \coloneqq \mu_1 \Restr{S_1}$, $\mu_1^2 \coloneqq \mu_1 \Restr{S_2}$ where $S_1$ and $S_2$ are connected sets satisfying $S_1 \cup S_2 = \supp \mu_1$, $\mu_1(S_1) = \mu_0(C_1)$, and $\mu_1(S_2) = \mu_0(C_2)$.   Similar strategies can be applied with more interleaved connected components, provided that the corresponding masses allow for it. 
\end{remark}

The key idea to approach the proof of \cref{th:main-1d} is as follows. If $v \in  C \cap L^\infty$ and $|v| > 0$, by \cite[Theorem 1.2.6]{MR1336820} (\cf also \cite{MR3849077}), there exists one and only one\footnote{~A (local-in-time) solution $\phi$ of \cref{eq:ode-1d} must exist, if  $v$ is continuous, by Peano's theorem; moreover, if $v$ is bounded, it can be extended globally-in-time. Let us sketch the proof of uniqueness. The function  $G(\psi)\coloneqq \int_{x}^\psi \frac{\d \xi }{v(\xi)}$ is of class $C^1$ and $G' \neq 0$; hence $G$ has a $C^1$ inverse. We then compute
\[
\partial_t G(\phi(t,x))=\frac{\partial_t \phi(t,x)}{v(\phi(t,x))}=1,
\]
which  yields $G(\phi(t,x))=t$ and thus $\phi(t,x)=G^{-1}(t)$ and uniqueness follows. 
} solution of \cref{eq:ode-1d} in the following sense: 
\begin{align}  \label{eq:sol-auto}
\begin{aligned}
    &\phi(\cdot, x) \in C^1((0, +\infty)) && \text{ for every $x \in \R$}, \\ 
&\int_{x}^{\phi(t,\,x)} \frac{1}{v(\xi)} \, \mathrm{d} \xi= t, && t >0.
\end{aligned}
\end{align}
Let us suppose that $\phi(1,\cdot) \equiv \T$. Then, owing to \cref{eq:sol-auto}, 
\begin{equation}
\label{eq:int_inv_v}
\int_x^{\T(x)} \frac{\d \xi}{v(\xi)} = 1.
\end{equation}
That is, a primitive of $1/v$ (\ie, $F$ such that $F'= 1/v$) solves \emph{Abel's functional equation}\footnote{~The functional equation \cref{eq:abel} is the \emph{equation of semi-conjugacy of $\T$ with the standard shift}; see \cite[Section 2.2.2, p. 16]{MR1994638}.} (introduced in \cite{abel1881}):
  \begin{align}\label{eq:abel}
     F(\T (x)) = F(x) + 1, \quad x \in {\supp \mu_0}.
 \end{align} 
Differentiating \cref{eq:int_inv_v} with respect to $x$ yields \emph{Acz\'el--Jabotinsky--Julia's equation} (introduced in \cite{zbMATH02611173,zbMATH03059802,zbMATH03184158}\footnote{~In the language of Jabotinsky and Acz\'el, the function $\T$ is the unknown and $v$ is given. On the other hand, in our setting, $\T$ is given and $v$ is the unknown.};  see also \cite{zbMATH06579010}):
 \begin{align}\label{eq:jj}
   v(\T (x)) =   \T '(x) \,  v(x), \quad x \in {\supp \mu_0}.
 \end{align}
 
 Viceversa, a solution $v \in  C \cap L^\infty$, with $|v|>0$, of \cref{eq:jj} generates a unique flow $\phi$ that satisfies $\phi(1,\cdot) \equiv \T$ (up to a scaling constant to achieve $\T$ at $t=1$). 

Therefore, to prove \cref{th:main-1d}, we will build a suitable solution $v$ to Acz\'el--Jabotinsky--Julia's equation \cref{eq:jj}, which is more convenient than \cref{eq:abel} for our purposes.

\subsection{Measures with bounded supports}
\label{ssec:proof-1d}

To solve \cref{eq:jj}, we distinguish various cases, according to the number of fixed points, denoted $\mathcal{S}$, of the optimal map $\T$ between $\mu_0$ and $\mu_1$, and the boundedness of the supports of the measures $\mu_0$ and $\mu_1$. 

In this section, we deal with compactly supported measures, so we will add the following condition: 
    \begin{enumerate}[label=\textbf{M-\arabic*}] 
     \setcounter{enumi}{2}
        \item \label{it:m4} $\supp \mu_0$ and $\supp \mu_1$ are compact.
        \end{enumerate}
        In particular, we will denote $\lambda, \Lambda>0$ the two constants such that 
\begin{equation}
\label{eq:LL}
0 < \lambda <\bar \mu_0, \, \bar \mu_1 < \Lambda < + \infty \quad \text{ in their respective supports} 
\end{equation}
(which exist, by compactness of the supports and continuity of the densities). 

We start with the case without fixed points. 

\begin{lemma}[Transport map without fixed points]
\label{lem:S_empty}
    Let us consider $\mu_0, \mu_1 \in \mathcal{P}_{\mathrm{a.c.}}(\R)$ and suppose that their densities, $\bar \mu_0$ and $\bar \mu_1$, are continuous functions (in their respective supports). Let us suppose that the conditions \ref{it:m1}--\ref{it:m3}--\ref{it:m4} hold (with \eqref{eq:LL}) and, moreover, that the map $\T$ has no fixed points, $\mathcal{S} = \emptyset$. 

    Then there exists   a continuous velocity field $v: {{\rm Conv}(\supp \mu_0 \cup \supp \mu_1)} \to \R$  such that 
    \[
|v|>0 \quad \text{in}  \quad {{\rm Conv}(\supp \mu_0 \cup \supp \mu_1)},\qquad     \T (x) = \phi(1,x),\quad\text{for}\quad   x \in {\supp \mu_0},
    \]
    where $\phi$ is the unique solution of \cref{eq:ode-1d} for $x \in { \supp \mu_0}$.
\end{lemma}
\begin{proof}
Let us denote 
\[
M_0 \coloneqq \supp \mu_0 = [a_0, b_0], \qquad M_1 \coloneqq \supp \mu_1 = [a_1, b_1].
\]
By previous considerations, such a velocity field $v$ exists if and only if it satisfies  \cref{eq:jj} (up to a multiplicative constant fixing the travel time). Recall that, from \cref{th:ot1d}, we already know that, for any $x\in M_0$,
\[
 \T'(x) = \frac{\bar\mu_0(x)}{\bar\mu_1(\T(x))},\quad\text{for}\quad x\in M_0,
\]
and in particular, by continuity of $\bar\mu_0$ and $\bar\mu_1$ (and boundedness away from zero for $\mu_1$), $\T'$ is continuous and $\T$ is  $C^1$. Moreover, by assumption \cref{it:m3} (recall \eqref{eq:LL}), 
\[
\frac{\Lambda}{\lambda}\ge \T'(x)\ge \frac{\lambda}{\Lambda}>0,\quad\text{for}\quad x\in M_0.
\]

We will split the proof into two cases. 

\textbf{Case 1:  ${M_0}\cap{M_1} = \emptyset$.} In this case, we can fix   $v \equiv  1$  in ${M_0}$, so that $v$ in ${M_1}$ is given by 
\[
v(x) \coloneqq \T'(\T^{-1}(x)) v(\T^{-1}(x)) = \T'(\T^{-1}(x)) \in \left[ \frac{\lambda}{\Lambda}, \frac{\Lambda}{\lambda}\right], \quad\text{for}\quad x\in{ M_1}. 
\]
In particular, $v$ can be chosen continuous and with $v(x)\in \left[\frac{\lambda}{\Lambda}, \frac{\Lambda}{\lambda}\right]$ for $x\in {\rm Conv}({M_0\cup M_1})$ to satisfy \cref{eq:jj}. That is, \cref{eq:int_inv_v} holds with a constant non-zero right-hand side. Up to multiplying by a constant to fix the transport time, we get the desired result. 

\textbf{Case 2:  ${M_0}\cap{M_1} \neq \emptyset$.} Let also assume, without loss of generality, that $a_0 < a_1$   (and therefore, $\T(x) > x$ for $x\in {M_0}$). Indeed, since $\T$ does not have fixed points, we already know that $a_0 \neq a_1$. Furthermore, if we had $a_1 < a_0$, we could swap the roles of $\mu_0$ and $\mu_1$ and consider the vector field $-v$ instead. 

Now we define $\alpha_0 \coloneqq a_0$, $\alpha_1 \coloneqq a_1 = \T(a_0) = \T(\alpha_0)$, and $\alpha_i \coloneqq \T(\alpha_{i-1})$ for $i=1,2,\dots$. Then, there exists $N \in \N$ such that $\alpha_N \in (b_0, b_1]$. Indeed, $i \mapsto \alpha_i$ is increasing (owing to the monotonicity of $\T$) and, if $\alpha_i \le b_0$, $\alpha_{i+1}\le b_1$. If the sequence $\{\alpha_i\}_{i}$ had an accumulation point $\bar{\alpha} \le b_0$, then $\T(\bar{\alpha}) = \bar{\alpha}$ and ${\alpha}$ is a fixed point for $\T$, which do not exist by assumption. Hence, the sequence must be finite. 

Let us now fix $v \in \left[\frac{\lambda}{\Lambda}, \frac{\Lambda}{\lambda}\right]$ to be any smooth function in $[a_0, a_1]$ with 
\begin{equation}
\label{eq:recurs_v0}
v(a_1) = \T'(a_0) v(a_0) = \frac{\bar\mu_0(a_0)}{\bar\mu_1(a_1)} v(a_0). 
\end{equation}
We then define, recursively, and denoting $\alpha_{N+1} \coloneqq b_1$,
\begin{equation}
\label{eq:recurs_v}
v(\T(x)) = \T'(x) v(x), \quad\text{for}\quad x\in [\alpha_i, \alpha_{i+1}],\ \  i = 0,1,\dots, N.
\end{equation}
This defines $v$ in the interval $[a_0, b_1]$ in a continuous way. Indeed, $v$ is continuous in $[\alpha_0, \alpha_1]$ and in $[\alpha_1, \alpha_2]$, and, owing to \crefrange{eq:recurs_v0}{eq:recurs_v}, is also continuous at $\alpha_1$ from both sides, thus being continuous in $[\alpha_0, \alpha_2]$. Then, 
\[
v(x) = \T'(\T^{-1}(x)) v(\T^{-1}(x)), \quad\text{for}\quad x\in [\alpha_1, \alpha_3]. 
\]
Since $\T^{-1}([\alpha_1, \alpha_3]) = [\alpha_0, \alpha_2]$, and $v$ is continuous in $[\alpha_0, \alpha_2]$, $\T^{-1}$ is continuous, and $\T'$ is continuous, we obtain that $v$ is continuous in $[\alpha_1, \alpha_3]$ as well, and thus, in $[\alpha_0, \alpha_3]$. Proceeding iteratively, it is continuous in the whole interval $[a_0, b_1]$  with a bound $v \le \left[\frac{\Lambda}{\lambda}\right]^{N+1}$. Up to multiplying by a constant to fix the time of transport as before, we get the desired result. 
\end{proof}

In the previous proof, higher regularity of $\bar \mu_0$ and $\bar\mu_1$ gives higher regularity for the velocity field $v$.

\begin{corollary}[Higher regularity]
\label{cor:Ck}
In the setting of \cref{lem:S_empty}, if, in addition, $\bar\mu_0\in C^k({\supp\mu_0})$ and $\bar \mu_1\in C^k({\supp\mu_1})$ for some $k\in \N$, then we can choose $v\in C^k({{\rm Conv}(\supp \mu_0 \cup \supp \mu_1)})$.
\end{corollary}
\begin{proof}
Recalling \cref{eq:derivative-t}, we note that  $\T'\in C^k({M_0})$ if $\bar\mu_0\in C^k({\supp\mu_0})$ and $\bar \mu_1\in C^k({\supp\mu_1})$. As a consequence, we deduce $\T\in C^{k+1}({M_0})$. Then, to conclude the proof, we just need to choose $v$ in $[a_0, a_1]$ as in the proof of \cref{lem:S_empty}, but such that 
\[
\frac{\d}{\d x^i}\bigg|_{x = a_1^-} v(x) = \frac{\d}{\d x^i}\bigg|_{x = a_0^+} (\T'(x) v(x)),
\]
 for all $i= 0,\dots,k$. Repeating the reasoning at the end of the proof of \cref{lem:S_empty}, we are done. 
\end{proof}

We now turn to the other cases, in which there can be fixed points. We start by assuming that $\T$ has exactly one fixed point.

\begin{lemma}[Transport map with exactly one fixed point]
\label{lem:S_1}
    Let us consider $\mu_0, \mu_1 \in \mathcal{P}_{\mathrm{a.c.}}(\R)$ and suppose that their densities, $\bar \mu_0$ and $\bar \mu_1$, are continuous functions (in their respective supports). Let us assume that the conditions \ref{it:m1}--\ref{it:m3}--\ref{it:m4}  hold (with \eqref{eq:LL}). Suppose, moreover, that the set $\mathcal{S}$ contains a single point, $\bar x$. 

    Then there exists   a  velocity field $v: {  \supp \mu_0 \cup \supp \mu_1 } \to \R$  such that 
   \[
    \begin{aligned}
    &|v|> 0  &&\text{in}  \quad {\supp \mu_0 \cup \supp \mu_1}\,\setminus \{\bar x\}, \\ 
    &v(\bar x) = 0, && {}
   \end{aligned}
    \]
and 
\[
    \T (x) = \phi(1,x), \quad x \in {\supp \mu_0},
    \]
    where $\phi$ is the unique solution of \cref{eq:ode-1d} for $x \in {\supp \mu_0}$, and $v$ is continuous except possibly at $\bar x$. If, moreover, $\bar\mu_0(\bar x) \neq \bar \mu_1(\bar x)$, then $v$ can be taken to be continuous also at $\bar x$. If, furthermore, $\bar \mu_0$ and $\bar \mu_1$ are Lipschitz continuous, then $v$ can be taken Lipschitz continuous up to $\bar x$ as well.  
\end{lemma}
\begin{proof}
We use the same notation as in the proof of \cref{lem:S_empty}. Let $\bar x$ be the unique fixed point for $\T$, and let us assume, without loss of generality,  that $\bar x = b_0 = b_1$ (otherwise, we can consider the restrictions of $\mu_0$ and $\mu_1$ to the intervals $(a_0, \bar x)$ and $(b_0, \bar x)$, in which, by assumption, they have the same mass---since $\bar x$ is a fixed point, $\mu_0((-\infty, \bar x)) = \mu_1((-\infty, \T(\bar x) = \bar x))$---; if $\bar x = a_0 = a_1$, instead, we can just flip the $x$-axis). 

We can furthermore assume $a_0 < a_1$  and thus $x < \T(x)$ for $x\in M_0$ (otherwise, we can exchange the roles of $\mu_0$ and $\mu_1$). Then, the sequence $\alpha_i$ constructed in the proof of \cref{lem:S_empty} is no longer finite, and $\alpha_i \to b_0 = b_1$ as $i \to +\infty$. This allows us to recursively define a (continuous) vector field $v$ in $(a_0, b_0)$ by means of \cref{eq:recurs_v}, after fixing it in $(a_0, \T(a_0))$ first. A priori, it could degenerate when approaching $\bar x$, though (cf. \cref{lm:ct}). 

If $\bar\mu_0(\bar x) \neq \bar\mu_1(\bar x)$, then we necessarily have $\T'(\bar x) < 1$ (because we are assuming $x < \T(x)$) and so, in the limit $i\to +\infty$, the sequence of intervals obtained recursively from \cref{eq:recurs_v} converges to $\bar x$:
\begin{equation}
\label{eq:vTi}
v(\T^i(x)) = v(x) { P_i(x)} \coloneqq v(x)  \prod_{j = 0}^{i-1} \T'(\T^j(x)),\quad \text{for}\quad x\in (a_0, \T(a_0)),
\end{equation}
and, since $\T'(\T^j(x)) < \frac12( 1+\T'(\bar x))< 1$ for $j$ large enough, we get 
\[
\|v\|_{L^\infty((\alpha_{i+1}, \alpha_{i+2}))}\le \|v\|_{L^\infty((\alpha_i, \alpha_{i+1}))}\to 0,\quad\text{as}\quad i \to +\infty,
\]
that is, $\lim_{x\to \bar x^-} v(x) = 0$, and we can fix $v(\bar x) = 0$. 

{
If $\bar \mu_0$ and $\bar \mu_1$ are Lipschitz continuous, by the same proof as in \cref{cor:Ck}, we immediately get that $v$ is locally Lipschitz continuous in $[a_0, \bar x)$. Let us check that, in fact, a bound on its derivative holds up to $\bar x$. 

Differentiating \cref{eq:vTi} and observing that $(\T^i)'(x) = \T'(\T^{i-1}(x))(\T^{i-1})'(x)=\dots = P_i(x)$, we get
\[
v'(\T^i(x)) = v'(x) +v(x) \sum_{j = 0}^{i-1} \frac{\T''(\T^j(x))}{\T'(\T^j(x))} P_j(x),\quad\text{for}\quad x\in (a_0, \T(a_0)).  
\]

In particular, since $\T' \ge \frac{\lambda}{\Lambda}$ and $\T''$ is bounded (because $\bar \mu_0$ and $\bar \mu_1$ are Lipschitz continuous), we can estimate 
\[
\begin{split}
&\|v'\|_{L^\infty((\alpha_i, \alpha_{i+1}))}\le \|v'\|_{L^\infty([a_0, \T(a_0)])} + C\|v\|_{L^\infty([a_0, \T(a_0)])} \sum_{j = 0}^{i-1}\tilde P_j,\\
&\text{where}\quad \tilde P_j\coloneqq \|P_j(x)\|_{L^\infty([a_0, \T(a_0)])}. 
\end{split}
\]
We observe that $P_j(x) = \T'(\T^{j-1}(x)) P_{j-1}(x)$. In particular, for $j$ large enough, $\T'(\T^{j-1}(x))\le 1-\eps$ with $\eps = \frac{1-\T'(\bar x)}{2}$, and we have $P_j(x)\le C (1-\eps/2)^j$ for all $j\in \N$, for some constant $C$. 

In conclusion, the previous sum is bounded and we get 
\[
\|v'\|_{L^\infty((\alpha_0, \bar x))}\le \|v'\|_{L^\infty([a_0, \T(a_0)])} + C\|v\|_{L^\infty([a_0, \T(a_0)])},
\]
that is, if $v$ is chosen smooth in $[a_0,\T(a_0)]$, we get a bound on $v'$ up to $\bar x$ and obtain that $v$ is Lipschitz continuous.
}

Finally, let us show the uniqueness of the flow. In the previous construction, we fixed $v(\bar x) = 0$.   Moreover, we have that, for any $\eps > 0$,
\[
\int_{A_{\pm, \eps}} \frac{\d x}{|v(x)|} = + \infty \quad\text{whenever}\quad A_{\pm, \eps} \neq \emptyset,
\]
where we introduced  the notation 
\[
A_{+, \eps} \coloneqq (\bar x, \bar x +\eps)\cap \supp \mu_0,\qquad A_{-, \eps} \coloneqq  (\bar x-\eps, \bar x)\cap \supp \mu_0.
\]

 Indeed, let us assume that we are in the same situation as in the construction above, and show an Osgood-type condition at $\bar x$:
\[
\int_{\bar x - \eps}^{\bar x } \frac{\d x}{|v(x)|} = +\infty. 
\]
Fix any $x_0\in (a_0, \bar x)$. We know that $\T^i(x_0) \uparrow \bar x$ as $i \to +\infty$, but also that 
\[
\int_{\T^i(x_0)}^{\T^{i+1}(x_0)} \frac{\d x}{|v(x)|} = 1. 
\]
In particular, taking $j\in \N$ large enough (depending on $\eps$) so that $\T^j(x_0)  > \bar x - \eps$, we get
\[
\int_{\bar x - \eps}^{\bar x } \frac{\d x}{|v(x)|} \ge \sum_{k \ge j} \int_{\T^k(x_0)}^{\T^{k+1}(x_0)} \frac{\d x}{|v(x)|} = \sum_{k \ge j} 1 = +\infty. 
\]

 We now show the uniqueness of the solution to \cref{eq:ode-1d} up to time 1 at all points $x\in {\supp \mu_0}$.  Since $v$ is continuous and non-zero in ${\supp \mu_0 \cup \supp \mu_1} \setminus \{\bar x \}$, we have existence and uniqueness for \cref{eq:ode} in this set (\cf \cite{MR2771257}). Indeed, for $x_0\in {\supp\mu_0}$ with $x_0\neq \bar x$, the flow can never cross $\bar x$ before time $1$ because, up until that moment, it would be continuous and, by the previous discussion, it requires infinite time to actually reach $\bar x$.

Now, we claim that the ODE
\[
\begin{cases}
      \partial_t  \phi(t,\bar x) = v(\phi(t,\bar x)), & t >0, \\
    \phi(0,\bar x) = \bar x,
\end{cases}
\]
is also well-posed: it has a unique solution $\phi(\cdot,\bar x) \equiv \bar x$. 

Indeed, arguing by  contradiction, let us assume, for example, that  $  \phi\left(t_1,\,\bar x\right)>\bar x$ for some $t_1>0$. For every $\delta>0$, we then have 
\[
t_1-\delta \geq \int^{t_1}_{\delta} \frac{\partial_t \phi(t,\,\bar x)}{v(\phi(t,\,\bar x))} \, \mathrm{d} t=\int_{\phi(\delta,\,\bar x)}^{\phi\left(t_1,\,\bar x\right)} \frac{1}{v(\xi)} \, \mathrm{d} \xi,
\]
which yields a contradiction because the 
\[
\limsup_{\delta \searrow 0} \int_{\phi(\delta,\,\bar x)}^{\phi\left(t_1,\,\bar x\right)} \frac{1}{v(\xi)} \, \mathrm{d} \xi = \int_{\bar x}^{\phi\left(t_1,\,\bar x\right)} \frac{1}{v(\xi)} \, \mathrm{d} \xi = \infty.
\]

We have thus shown that \cref{eq:ode-1d} has a unique  solution up to time $1$ at all points $x\in {\supp \mu_0}$
\end{proof}

Next, we deal with the case when $\T$ has two fixed points.

\begin{lemma}[Transport map with exactly two fixed points]
\label{lem:S_2}
    Let us consider $\mu_0, \mu_1 \in \mathcal{P}_{\mathrm{a.c.}}(\R)$ and suppose that their densities, $\bar \mu_0$ and $\bar \mu_1$, are continuous functions compactly supported in $[0, 1]$. Let us assume that the conditions \ref{it:m1}--\ref{it:m3}--\ref{it:m4}  hold (with \eqref{eq:LL}). Suppose, moreover, that the set $\mathcal{S}= \{0, 1\}$.

    Then there exists   a  velocity field $v: [0, 1]\to \R$  such that 
    \[
    \begin{aligned}
    &|v|> 0  &&\text{in}  \quad (0,1),\\ 
    &v(0) = v(1) = 0, && {}
   \end{aligned}
    \]
and 
\[
    \T (x) = \phi(1,x), \qquad x \in [0,1],
    \]
    where $\phi$ is the unique solution of \cref{eq:ode-1d} for $x \in {\supp \mu_0}$, and $v$ is continuous except possibly at $\mathcal{S}$. If, moreover, $|\bar\mu_0- \bar \mu_1|> 0$ in $\mathcal{S}$, then $v$ can be taken to be continuous also at $\mathcal{S}$. If, furthermore, $\bar \mu_0$ and $\bar \mu_1$ are Lipschitz continuous, then $v$ can be taken Lipschitz continuous up to $\mathcal{S}$. 
\end{lemma}
\begin{proof}
Up to swapping the roles of $\mu_0$ and $\mu_1$ (and changing the sign of the vector field $v$), we can assume $\T(x) < x$ for $x \in (0, 1)$. 

Let $\nu_0$ and $\nu_1$ be the restrictions of $\mu_0$ and $\mu_1$, respectively, in the intervals $(0, 1/2)$ and $(0, p)$, where $p \coloneqq \T(1/2) < 1/2$. By definition of $\T$, $\nu_0$ and $\nu_1$ still have the same mass. Moreover, since $p < \frac12$, $\T\Restr{\supp \nu_0}$ has only one fixed point (because the support of $\nu_0$ is $[0, 1/2]$, and $\T(x) < x$ in $(0, 1)$), namely, $0$. Thus, we can apply \cref{lem:S_1} and deduce that there exists some $v$ defined in $[0, \frac12]$ that is continuous (except, possibly,  at $0$) and such that $\nu_0$ is transported to $\nu_1$. Such a velocity field is arbitrarily defined in $[p, \frac12]$ and then extended to the whole $[0, 1/2]$ by means of \cref{eq:int_inv_v}. The compatibility condition that needs to be satisfied to get continuity is given by 
\[
v(p) = \T'(1/2)\, v(1/2).
\]

On the other hand, let $\tilde \nu_0$ and $\tilde \nu_1$ be, respectively, the restrictions of $\mu_0$ and $\mu_1$ to the intervals $(1/2, 1]$ and $(p, 1]$. Then, we can repeat the proof of \cref{lem:S_1} with $\mu_0 = \tilde \nu_1$ and $\mu_1 = \tilde \nu_0$, which have $\T^{-1}$ as monotone transport map, to deduce again the existence of a continuous velocity field $\tilde v$ in $[p, 1]$ determined from its value on the interval $[p, 1/2]$ through  \cref{eq:int_inv_v}. The compatibility condition now is
\[
\tilde v(1/2) = (\T^{-1})'(p) \tilde v(p) = \frac{1}{\T'(\T^{-1}(p))} \tilde v(p) = \frac{1}{\T'(1/2)}\tilde v(p).  
\]
Namely, we can take $\tilde v = -v$ on $[p, 1/2]$. Then, the velocity field $v$ transports $\nu_0$ to $\nu_1$ and is continuous in $[0, 1/2]$, and the velocity field $-\tilde v$ transports $\tilde\nu_1$ to $\tilde \nu_0$ and is continuous in $[p, 1]$. Since $v = -\tilde v$ in $[p, 1/2]$, we can continuously extend $v$ by $-\tilde v$ to the whole interval $[0, 1]$. In particular, in this construction, we fix $v(0) = v(1) = 0$.  

As in the proof of \cref{lem:S_1}, we can show that, for any $\eps > 0$, 
\[
\int_{0}^{\eps} \frac{\d x}{|v(x)|} = \int_{1-\eps}^{1} \frac{\d x}{|v(x)|} =  +\infty. 
\]
and that existence and uniqueness of solutions of \cref{eq:ode-1d} hold  also at the fixed points.
\end{proof}

\subsection{Measures with unbounded supports} 

For measures with unbounded supports, we can recover results that are analogous to \cref{lem:S_empty}, \cref{cor:Ck}, and \cref{lem:S_1}. 

\begin{lemma}[Transport map without fixed points---unbounded setting]
\label{lem:S_empty_ub}
     Let us consider $\mu_0, \mu_1 \in \mathcal{P}_{\mathrm{a.c.}}(\R)$ and suppose that their densities, $\bar \mu_0$ and $\bar \mu_1$, are continuous functions (in their respective supports). Let us assume that the conditions \crefrange{it:m1}{it:m3} hold,  that either the support of $\mu_0$ or the support of $\mu_1$ is unbounded, and that the map $\T$ has no fixed points, $\mathcal{S} = \emptyset$.  

    Then there exists   a continuous velocity field $v:  {{\rm Conv}(\supp \mu_0 \cup \supp \mu_1)} \to \R$  such that 
    \[
|v|>0 \quad \text{in}  \quad  {{\rm Conv}(\supp \mu_0 \cup \supp \mu_1)},\qquad     \T (x) = \phi(1,x),\quad\text{for}\quad   x \in  {\supp \mu_0},
    \]
    where $\phi$ is the unique solution of \cref{eq:ode-1d} for $x \in  { \supp \mu_0}$.
\end{lemma}
\begin{proof}
    Let us suppose, first, that ${\rm supp}(\mu_0) = [a_0, b_0]$ and ${\rm supp}(\mu_1) = [a_1, \infty]$, where $a_0,a_1\in \R$, $b_0\in \R\cup\{\infty\}$. 
Up to switching the roles of $\mu_0$ and $\mu_1$ if $b_0 = \infty$, we can assume $a_0 < a_1$ and $\T(x) > x$ for all $x\in {\rm supp}(\mu_0)$. 
    
The proof follows as the one in \cref{lem:S_empty} (we use the same notation, as well). In this case, we have that $\T'$ is continuous and positive in $M_0$, but we do not have universal bounds for it.

The case $M_0\cap M_1 = \emptyset$, follows exactly as \cref{lem:S_empty}, without uniform controls on the velocity (which might blow-up or go to zero at infinity). 

In the case $M_0\cap M_1 \neq \emptyset$ we define $\alpha_i$ again as in \cref{lem:S_empty}, and construct $v$ (continuous, and positive) in the same way recursively in the intervals $[\alpha_i, \alpha_{i+1}]$ (basically, $v$ is arbitrarily fixed in $[a_0, a_1]$, and then uniquely continued).  Differently from before, however, we lose the global $L^\infty$ control on the velocity, and in this case, $N$ might  even be infinite (when $b_0 = \infty$). 

We suppose now $a_0 = -\infty \le a_1$. Let any $c_0\in (-\infty, b_0)$, $c_1 = \T(c_0) > c_0$, and consider $\nu_0$ and $\nu_1$ the restrictions of $\mu_0$ and $\mu_1$ to the intervals $[c_0, b_0]$ and $[c_1, \infty]$. Then, by the previous argument, we can construct a velocity field in $[c_0, \infty]$ transporting $\nu_0$ into $\nu_1$. 
On the other hand, if $\bar\nu_0$ and $\bar \nu_1$ are the restrictions of $\mu_0$ and $\mu_1$ to the intervals $[-\infty, c_0]$ and $[a_1, c_1]$, then the previous velocity field (defined, for these measures, in $[c_0, c_1]$) uniquely extends, by the previous arguments, to the whole interval $[-\infty, c_1]$ as well. 
\end{proof}

\begin{corollary}[Higher regularity---unbounded setting]
\label{cor:Ck_ub}
In the setting of \cref{lem:S_empty_ub}, if, in addition, $\bar\mu_0\in C^k( {\supp\mu_0})$ and $\bar \mu_1\in C^k( {\supp\mu_1})$ for some $k\in \N$, then we can choose $v\in C^k( {{\rm Conv}(\supp \mu_0 \cup \supp \mu_1)})$.
\end{corollary}
\begin{proof}
Cf. proof of \cref{cor:Ck} by means of \cref{lem:S_empty}, using \cref{lem:S_empty_ub} in this case. 
\end{proof}

\begin{lemma}[Transport map with exactly one fixed point---unbounded setting]
\label{lem:S_1_ub}
     Let us consider $\mu_0, \mu_1 \in \mathcal{P}_{\mathrm{a.c.}}(\R)$ and suppose that their densities, $\bar \mu_0$ and $\bar \mu_1$, are continuous functions (in their respective supports). Let us assume that the conditions \crefrange{it:m1}{it:m3} hold.

        Let us suppose, moreover, that ${\rm supp}(\mu_0) = [a_0, b_0]$ and ${\rm supp}(\mu_1) = [a_0, \infty]$, where $a_0\in \R$, $b_0\in \R\cup\{\infty\}$, and that the map $\T$ has a single fixed point, $\mathcal{S} = \{a_0\}$. 

    Then there exists   a  velocity field $v:  {  \supp \mu_0 \cup \supp \mu_1 } \to \R$  such that 
   \[
    \begin{aligned}
    &|v|> 0  &&\text{in}  \quad  {\supp \mu_0 \cup \supp \mu_1}\,\setminus \{a_0\}, \\ 
    &v(a_0) = 0, && {}
   \end{aligned}
    \]
and 
\[
    \T (x) = \phi(1,x), \quad x \in  {\supp \mu_0},
    \]
    where $\phi$ is the unique solution of \cref{eq:ode-1d} for $x \in {\supp \mu_0}$, and $v$ is continuous except possibly at $a_0$. If, moreover, $\bar\mu_0(a_0) \neq \bar \mu_1(a_0)$, then $v$ can be taken to be continuous also at $a_0$. If, furthermore, $\bar \mu_0$ and $\bar \mu_1$ are Lipschitz continuous, then $v$ can be taken Lipschitz continuous up to $a_0$.  
\end{lemma}
\begin{proof}
We can assume, without loss of generality, that $\T(x) > x$ (which is directly true if $b_0 < \infty$, since there is only one fixed point). 

Take any $c_0 \in (a_0, b_0)$, $c_1 = \T(c_0) > c_0$. We can now apply \cref{lem:S_1} to $\nu_0$ and $\nu_1$, the restrictions of $\mu_0$ and $\mu_1$ respectively in the intervals $[a_0,  c_0]$ and $[a_0, c_1]$, which will fix a velocity field in $[a_0, c_1]$. 

Let $\bar\nu_0$ and $\bar \nu_1$ be the restrictions of $\mu_0$ and $\mu_1$ in $[c_0, b_0]$ and $[c_1, \infty]$ respectively. Notice that, by assumption, the monotone map from $\bar\nu_0$ to $\bar \nu_1$ does not have any fixed point. We can now apply the result from \cref{lem:S_empty_ub}, where the velocity field has already been fixed in the interval $[c_0, c_1]$, and can be uniquely extended as in the proof of \cref{lem:S_empty_ub} (and \cref{lem:S_empty}) to the whole space. 
\end{proof}

\section{Proofs of the main results}

Finally, putting everything together we can deal with the general case,   \cref{th:main-1d}.

\begin{proof}[Proof of \cref{th:main-1d} (and, therefore, \cref{th:A})]
If $\mathcal{S} = \emptyset$, we apply \cref{lem:S_empty}  (or \cref{lem:S_empty_ub}); if $\mathcal{S} = \{x_1\}$, we apply \cref{lem:S_1} or \cref{lem:S_1_ub} (on each side, $(-\infty, x_1)$ and $(x_1, +\infty)$), and if $\mathcal{S} = \{x_1, x_2\}$ with $x_1  < x_2$, we apply \cref{lem:S_1} (or \cref{lem:S_1_ub}) to the side intervals (namely, $(-\infty, x_1)$ and $(x_2, +\infty)$), and \cref{lem:S_2} to the middle interval (that is, $(x_1, x_2)$), after a rescaling and translation if necessary. 

Let us, therefore, suppose that $\mathcal{S}$ contains more than two elements.  We fix $v \equiv 0$ in ${\mathcal{S}}$. 
Since $\T$ is continuous,  $ \R\setminus \mathcal{S}$ can be written as a disjoint countable union of open intervals (being an open set in $\R$), $  \R\setminus \mathcal{S} = \bigcup_{i\in \N} I_i$. We can fix $I_0 \coloneqq (-\infty, x_l)$ if  $x_l \coloneqq \min(\mathcal{S}) > -\infty$ and $I_1 \coloneqq (x_r, +\infty)$ if $x_r \coloneqq \max(\mathcal{S})<\infty$ (them being empty otherwise). We then apply \cref{lem:S_1} or \cref{lem:S_1_ub} to $I_0$ and $I_1$, and \cref{lem:S_2} to each $I_i$ with $i\ge 2$, to get the desired result.

In order to obtain the conditional continuity and Lipschitz regularity, we notice that, if $|\bar \mu_0 - \bar \mu_1| > 0$ on $\partial\mathcal{S}$, then we  cannot have an accumulation point of $\partial\mathcal{S}$ (alternatively, in any compact set, there are finitely many intervals $I_i$). Indeed, otherwise there would be a sequence of points $x_k\in \partial I_i \subset \mathcal{S}$ with $x_k \to \bar x\in \mathcal{S}$ as $k\to +\infty$. In particular, from $\T(x_k) = x_k$, we get 
\[
\T'(\bar x) = \frac{\T(x_k) - \T(\bar x)}{x_k - \bar x} = 1;
\]
in turn, by \cref{eq:ode-1d}, this implies $\bar\mu_0(\bar x) = \bar\mu_1(\bar x)$,  contradicting our assumption. Hence, we can apply \cref{lem:S_2} finitely many times in any given compact set, and use the fact that the concatenation of finitely many continuous or Lipschitz continuous functions remains continuous or Lipschitz continuous, to obtain that $v$ can be taken locally Lipschitz continuous in this case. 

Finally, on the uniqueness of the flow, as in the proof of \cref{lem:S_1}, we can show that, for any  $\bar x\in \partial\mathcal{S}$ and $\eps > 0$,
\[
\int_{A_{\pm, \eps}} \frac{\d x}{|v(x)|} = + \infty \quad\text{whenever}\quad A_{\pm, \eps} \neq \emptyset,
\]
where we use again the notation 
\[
A_{+, \eps} \coloneqq (\bar x, \bar x +\eps)\cap \supp \mu_0,\qquad A_{-, \eps} \coloneqq  (\bar x-\eps, \bar x)\cap \supp \mu_0.
\]
Indeed, if $A_{\pm, \eps}\cap \mathring{\mathcal{S}} \neq \emptyset$, since $v \equiv 0$ in $\mathcal{S}$, the result follows. There are now two cases.

\begin{enumerate}[label=\textbf{Case \arabic*.}]

\item If $A_{\pm, \eps}\cap \partial \mathcal{S} = \emptyset$ for $\eps > 0$ small enough, we are in the same situation as \cref{lem:S_1}, so the result follows. 
\item  If $A_{\pm, \eps}\cap \partial \mathcal{S} \neq \emptyset$ for all $\eps>0$ small, then there is a monotone sequence of fixed points $\partial\mathcal{S}\ni \bar x_k \to \bar x$ as $k\to \infty$, with $x_k \in A_{\pm, \eps}$. Let us assume, without loss of generality, that $x_k$ is an increasing sequence and we are looking at $A_{-, \eps}$. Notice that, if $y_k\in (x_k, x_{k+1})$ is not a fixed point (which we can always find, otherwise $(x_k, x_{k+1})$ is an open interval contained in $A_{\pm, \eps}$), then $\T(y_k) \in (x_k, x_{k+1})$ as well, with $\T(y_k) \neq y_k$. Then, we have 
\[
\int_{A_{-, \eps}}\frac{\d x}{|v(x)|} \ge \sum_{k\in \N} \left|\int_{y_k}^{\T(y_k)} \frac{\d x}{v(x)}\right| = \sum_{k\in \N} 1 = +\infty. 
\]
\end{enumerate}
In both cases, we get the desired result. 

 The existence and uniqueness of solutions to \cref{eq:ode-1d} now follow arguing as in \cref{lem:S_1}.

\end{proof}

\begin{remark}[On the assumption ``$|\bar\mu_0- \bar \mu_1|> 0$ in $\mathcal{S}$''] \label{rmk:preva}
While the condition $|\bar\mu_0 - \bar \mu_1|> 0$ in $\mathcal{S}$ can be violated for a suitable choice of measures $\mu_0$ and $\mu_1$, such a situation is \emph{uncommon}. Namely, given $\mu_0$, by \cref{th:main-1d} we have that, for ``almost every'' $\mu_1$, we can construct a Lipschitz continuous field $v$. This can be formalized as follows. 

Given a measure $\mu_0$ and a monotone map $\T$, we define $\mu_1 \coloneqq \T_\# \mu_0$. Let us show that, given any measure $\mu_0$ satisfying \crefrange{it:m1}{it:m3}, the set of strictly monotone and $C^1$ maps $\T$ for which $\mu_0$ and $\mu_1 \coloneqq \T_{\#}\mu_0$   satisfy $|\bar \mu_0 - \bar \mu_1| > 0$ in $\mathcal{S} = \mathcal{S}(\T)$ is \emph{relatively prevalent}\footnote{~We use  the language of \emph{prevalence} from~\cite{MR1161274,MR1191479,MR2149086}.} among monotone and $C^1$ maps (we use the notion of relative prevalence on the set of strictly monotone maps seen as a convex subset of $C^1$ maps). Namely, \emph{for almost every} map $\T$, we have $|\bar \mu_0 - \bar \mu_1| > 0$ in $\mathcal{S}$.

To show that, given any $\mu_0$ and $\T$ as above, and $f\in C^1$, we construct the following map (\cf~\cite[Definition 4.1 or 4.5]{MR2149086}):
\[
\T_{\lambda, f}(x) \coloneqq \T(x) - f(x) -\lambda,\quad\text{for}\quad x\in \R, \ \lambda\in \R. 
\]
It is now enough to prove that (by \cite[Proposition 4.6]{MR2149086}, see also \cf~\cite[Definitions 4.1 or 4.5]{MR2149086}), on the one hand $|\{\lambda\in \R : \T'_{\lambda,\bar f}> 0\}| > 0$ for some $\bar f\in C^1$ (which immediately holds taking, for example, $\bar f(t) = -t$); and on the other hand, that for almost every $\lambda\in \R$ and every $f\in C^1$ with $(\T - f)'> 0$, 
\[
|\bar \mu_0  - \bar \mu_{\lambda, f}| > 0\quad\text{in}\quad \mathcal{S}_{\lambda, f} \coloneqq \mathcal{S}(\T_{\lambda, f}) = \{x : x = \T_{\lambda, f}(x)\},
\]
where $\mu_{\lambda, f} = (\T_{\lambda, f})_{\#} \mu_0$ and $\bar\mu_{\lambda, f} $ is its density as an absolutely continuous measure. Equivalently, by \cref{th:ot1d}, it is enough to show that 
\[
\T_{\lambda, f}'(x) \neq 1\quad\text{in}\quad \mathcal{S}(\T_{\lambda, f})\quad \Leftrightarrow \quad \widetilde \T_f' \neq 0 \quad\text{in}\quad \{\widetilde \T_f = \lambda\},\qquad  \quad\text{for a.e.}\ \lambda\in \R, 
\]
where we have denoted $\widetilde \T_f \coloneqq \T - f-{\rm Id}$. This immediately holds by Sard's theorem.
\end{remark}

From \cref{rmk:preva}, we deduce the following corollary of \cref{th:main-1d}.

\begin{corollary}[Approximate controllability, $d=1$]\label{prop:approx}
        Let us consider $\mu_0, \mu_1 \in \mathcal{P}_{\mathrm{a.c.}}(\R)$ and suppose that their densities, $\bar \mu_0$ and $\bar \mu_1$, are continuous functions (in their respective supports). Let us assume that \crefrange{it:m1}{it:m3} hold. 
        For every $\varepsilon >0$, there exists $\mu_1^\varepsilon \in \mathcal{P}_{\mathrm{a.c.}}(\R)$ such that $\mathrm{dist}(\mu_1, \,  \mu^\varepsilon_1)  < \varepsilon$  (in the sense of the $L^1$ or of the Wasserstein distance) and there exists a continuous velocity field $v^\varepsilon: {{\rm Conv}(\supp \mu_0 \cup \supp \mu_1^\varepsilon)} \to \R$  such that 
   \[
    \begin{aligned}
   & |v^\varepsilon|> 0  &&\text{in}  \quad {{\rm Conv}(\supp \mu_0 \cup \supp \mu_1)}\,\setminus \mathcal{S}, \\ 
    &v^\varepsilon \equiv 0 &&\text{in}\quad {\mathcal{S}},
   \end{aligned}
    \]
    where  $\mathrm T$ is the monotone optimal transport map from \cref{th:ot1d},  $\mathcal S$ is the set of fixed points of the map $\T$ in ${\supp \mu_0}$, and $\phi$ is the unique solution of  \cref{eq:ode-1d}. If,  furthermore, $\bar \mu_0$ is Lipschitz continuous, then  $v^\varepsilon$ can be taken Lipschitz continuous.
\end{corollary}

\cref{prop:approx} implies, in particular, \cref{th:B} as well.

\begin{proof}[Proof of \cref{prop:approx} (and, therefore,  \cref{th:B})]
Up to a standard smoothing argument, we can assume that $\bar \mu_1$ is Lipschitz. We consider the monotone map  $\T$ between $\mu_0$ and $\mu_1$ and define $\mu_1^\eps \coloneqq (\T(\cdot) - \lambda)_\# \mu_0$ for some arbitrarily small $\lambda$ such that we are in the context of \cref{rmk:preva}. Then the result follows from \cref{th:main-1d}. 
\end{proof}

\subsection{On the assumptions of  \texorpdfstring{\cref{th:main-1d}}{Theorem 2.2}}
\label{ssec:sharp}

We conclude this section by showing that, in \cref{lem:S_1} (and, consequently, in   \cref{th:main-1d}), if $\bar \mu_0$ and $\bar \mu_1$ have the same value at the fixed point $\bar x$, continuity of $v$ may indeed fail. Moreover, in general, $v$ does not need to belong to $L^1_\loc$ around $\bar x$.

\begin{lemma}\label{lm:ct}
There exist measures $\mu_0$ and $\mu_1$  satisfying the hypotheses of \cref{lem:S_1}  such that $\bar \mu_0(\bar x) = \bar\mu_1(\bar x)$  and either $v$ cannot be taken bounded, or there is no uniqueness of the flow \cref{eq:ode-1d} (more precisely, it is not true $|v| > 0$ outside of $\mathcal{S}$).
Moreover, if $v$ is continuous outside of $\bar x$, then it does not belong to $L^1_{\rm loc}$ around $\bar x$.
\end{lemma}

Before proving \cref{lm:ct}, we need the following preliminary result. 

\begin{lemma}\label{lem:Tcount2}
There exists a map $\T:[0, 1]\to [0, 1]$ such that 
\[
\T\in C^1([0, 1]),\qquad  \frac12 \le \T'\le \frac32 \quad\text{in}\quad [0, 1],\qquad \T(x) < x,\quad\text{for}\quad x\in (0, 1],
\]
and  
\[
\prod_{i\ge 0} \T'(\T^i(1/2)) = +\infty. 
\]
\end{lemma}
\begin{proof}

Let us first construct a map $\rS\in C^1([0, 1])$ such that $\rS(0) = 0$, $\rS > 0$ in $(0, 1)$, $-\frac12 \le \rS' \le \frac12$, and $\rS'(\T^i(1/2)) = \gamma_i$. By taking \[\T(x) \coloneqq x - \rS(x),\] we get our result. 

We fix first the points $\T^i(1/2) \coloneqq \alpha_i$, taking $\alpha_1  < \frac12$ and $i\mapsto \alpha_i$ strictly decreasing and converging to zero. At each of these points, we then  fix a value $\beta_i = \rS(\alpha_i)$ such that the following compatibility condition holds: 
\begin{equation}
\label{eq:alphai2}
\alpha_{i+1} = \alpha_i - \beta_i,\qquad\beta_{i+1}<  \beta_i,\qquad\beta_i \downarrow 0\quad\text{as}\quad i \to +\infty.
\end{equation}
Now, the map $\rS$ is defined at each point $\alpha_i$. Let us fix it between in the interval $[\alpha_{i+1}, \alpha_i]$ for each $i\in \N$. To do that, let, for $0<\bar\gamma <\frac12 $,
\[
\begin{aligned}
 &\varphi_{\bar \gamma} \in C^\infty([0, 1]) ,\quad\text{such that } -\frac14 \le\varphi_{\bar\gamma}\le \frac54, \\ &\text{with } \ \varphi_{\bar\gamma}'(0) = -\bar \gamma,\quad  \varphi_{\bar\gamma}'(1) = -\frac14, \quad \varphi_{\bar\gamma}(0) = 0, \quad \varphi_{\bar\gamma}(1) = 1,\quad -\frac12 \le \varphi_{\bar\gamma}'\le \frac32.    
\end{aligned}
\]
We then define, for some $0<\bar\gamma_i<\frac12 $ to be chosen, 
\[
\rS(x)  \coloneqq  \beta_{i+1}+ (\beta_i- \beta_{i+1}) \varphi_{\bar\gamma_i}\left(\frac{x-\alpha_{i+1}}{\alpha_{i}-\alpha_{i+1}}\right),\quad\text{for}\quad x\in (\alpha_{i+1}, \alpha_i). 
\]
To ensure $\rS\in C^1((0, 1/2])$, we impose $\frac{\beta_i - \beta_{i+1}}{4(\alpha_i-\alpha_{i+1})} = \bar\gamma_i \frac{\beta_{i-1} - \beta_{i}}{(\alpha_{i-1}-\alpha_{i})} $, that is, $\bar\gamma_i = \frac{\beta_{i-1}(\beta_i - \beta_{i+1})}{4\beta_i (\beta_{i-1}- \beta_i)}$ from \eqref{eq:alphai2}, and have $\rS'(\alpha_i) = -\frac{\beta_i - \beta_{i+1}}{4\beta_i} $ for all $i\ge 0$. Moreover, if we want to bound $\T'$, we need
\[
-\frac12\le -\frac12 \frac{\beta_i - \beta_{i+1}}{\beta_i}\le \rS'(x) \le \frac32 \frac{\beta_i - \beta_{i+1}}{\alpha_i -\alpha_{i+1}} = \frac32 \frac{\beta_i - \beta_{i+1}}{\beta_i} \le \frac12 \quad\text{in}\quad [\alpha_{i+1}, \alpha_i];
\]
if we also want $\T'$ to be continuous at $0$, we necessarily need $\rS'(x)\to 0$ as $x\to 0$, or, 
\[
\frac{\beta_i - \beta_{i+1}}{\beta_i}\to 0,\quad\text{as}\quad i \to +\infty. 
\]

Finally, we notice that, to have $\rS > 0$ in $(0, 1/2]$, we need $\beta_{i+1}-\frac14 (\beta_i - \beta_{i+1})  >0$, that is, $\frac{\beta_{i+1}}{\beta_i} > \frac15$.

In conclusion, we just need to construct a decreasing sequence $\beta_i$ such that, from \cref{eq:alphai2}   and the considerations above,
\[
\sum_{i \ge 0} \beta_i = \alpha_0 = \frac12,\quad \frac{\beta_{i+1}}{\beta_i} \ge \frac23\quad\text{for}\quad i \ge 0, \quad   \frac{\beta_{i+1}}{\beta_i}\to 1,\quad\text{as}\quad i \to +\infty, 
\]
and
\[
\bar\gamma_i = \frac{\beta_{i-1}(\beta_i - \beta_{i+1})}{4\beta_i (\beta_{i-1}- \beta_i)} < \frac12.
\]

We can take, for example, $\beta_i \coloneqq  \frac{\gamma}{(i+10)^2}$, where $\gamma \coloneqq \frac12 (\sum_{j \ge 10} j^{-2})^{-1}$. Then, $\rS$ is a  map with $\rS\in C^1([0, 1/2])$, such that $\rS(0) = 0$, $\rS > 0$ in $(0, 1)$, $-\frac12 \le \rS' \le \frac12$, $\rS'(0) = 0$, and $\rS'(\T^i(1/2)) = -\frac{\beta_i-\beta_{i+1}}{4\beta_i}$. 

In this situation, we have that $\T(x) = x - \rS(x)$ satisfies 
\[
\T( \alpha_i ) =\alpha_i - \rS(\alpha_i) = \alpha_i - \beta_i = \alpha_{i+1},
\]
$\T\in C^1([0, 1/2])$, with $\T'(x) = 1- \rS'(x) \in [1/2, 3/2]$, and $\T'(\alpha_i) = 1- \rS'(\alpha_i) = 1+\frac{\beta_i - \beta_{i+1}}{4\beta_i}$; therefore, 
\[
\prod_{i \ge 0} \T'( \alpha_i) \ge \frac14 \sum_{i \ge 0} \left(1- \frac{\beta_{i+1}}{\beta_i}\right) = +\infty,
\]
since $1-\frac{\beta_{i+1}}{\beta_i} = 1 - \frac{(i+10)^2}{(i+11)^2} = \frac{2i+21}{(i+11)^2}$.
 By extending the map to $[1/2, 1]$ as needed, we get the desired result.
\end{proof}

Thanks to the previous construction, we can prove that $v$ is not regular or even integrable in general.

\begin{proof}[Proof of \cref{lm:ct}]
Let $\mu_0$  be the uniform measure in $[0, 1/2]$ and let $\T$ be the map constructed in \cref{lem:Tcount2}. Let $\mu_1 \coloneqq \T_\#\mu_0$. Then $\mu_0$ and $\mu_1$ satisfy the hypotheses of \cref{lem:S_1}. 

Let us suppose now that there exists a vector field $v$  defined in $[0, 1/2]$ inducing a unique flow, with $v \le 0$ in $(0, 1/2]$, transporting $\mu_0$ to $\mu_1$. By assumption, $v(0) = 0$ (since $0$ must be transported to $0$), and $v(1/2) < 0$ (since $\T(1/2) < 1/2$ and there is uniqueness of the flow). Then, from \cref{eq:vTi},  we deduce 
\[
v(\T^i(1/2)) = v(1/2) \prod_{j = 0}^{i-1}\T'(\T^j(1/2)) \to - \infty\quad\text{as}\quad i \to \infty,
\]
so $v$ is not bounded around $0$. This proves the first part of the result. 

For the second part, we continue from the construction in \cref{lem:Tcount2}, using the same notation. Let us assume, in such a construction, moreover, that $\varphi'(t) = -\frac14$ for all $t\in [\frac{9}{10}, 1]$. We have that, if $\bar\alpha\in (\alpha_{i+1}, \alpha_i)$ is such that $\alpha_i - \bar \alpha < \frac{1}{10}\beta_i$, then
 \[
 \T^{-1}(\alpha_i) - \T^{-1}(\bar \alpha) = \frac{\alpha_i -\bar \alpha}{1+\frac{\beta_i -\beta_{i+1}}{4\beta_i}}< \frac{1}{10} \frac{\beta_i}{1+\frac{\beta_i - \beta_{i+1}}{4\beta_i}}< \frac{1}{10}\beta_{i-1}
 \]
 whenever $(\beta_i)_i$ is decreasing.

 Hence, we have that, for any $\eta\in (0, \frac{1}{10}]$,
 \[
 \begin{aligned}
  v(x) &\ge \prod_{j = 0}^i \T'(\alpha_i)\inf_{\xi\in (1/2 - \eta \beta_0, 1/2]} v(\xi) \\ & \ge \frac14 \sum_{j = 0}^i \frac{\beta_i - \beta_{i+1}}{\beta_i}\inf_{\xi\in (1/2 - \eta \beta_0, 1/2]} v(\xi),\quad\text{for all} \quad x\in (\alpha_i - \eta \beta_i, \alpha_i].
  \end{aligned}
 \]

 Let us take now $\beta_i = f(i) \coloneqq \frac{\gamma}{i\log^2 i}$ for a suitable universal constant  $\gamma$. Then we have, on the one hand, that 
 \[
 \alpha_i \asymp \int_i^\infty f(s)\, \mathrm d s \asymp \frac{1}{\log i},\quad\text{for}\quad i \ge 2, 
 \]
 where $\asymp$ denotes comparable quantities by universal constants. 
 On the other hand, 
 \[
 \begin{aligned}
 C \sum_{j = 0}^i \frac{\beta_i - \beta_{i+1}}{\beta_i} &\ge  -\sum_{j = 0}^i \frac{f'(i)}{f(i)} = \sum_{j = 0}^i |\log(f)'(i) | \\ &\asymp |\log(f(i))| = |\log(\gamma e^{-\frac{1}{\alpha_i}}\alpha_i^2) |= \frac{1}{\alpha_i} - C - 2\log(\alpha_i).
 \end{aligned}
 \]

 We always fix $\inf_{\xi\in (1/2 - \eta \beta_0, 1/2]} v(\xi) > c_0 > 0$ by continuity of $v$ (this is the only place where continuity is used, to say $\inf_{\xi\in (1/2 - \eta \beta_0, 1/2]} v(\xi) >  0$ for some $\eta  >0$). Then, from the computations above, we deduce
 \[
 v(x) \ge c_0 \left(\frac{1}{\alpha_i} - C - 2\log(\alpha_i)\right),\quad\text{for all}\quad x\in (\alpha_i - \eta \beta_i, \alpha_i]. 
 \]
 In particular, since $\beta_i/\beta_{i+1}\approx 1$ for $i $ large, $v$ does not belong to $L^1$.
\end{proof}

\begin{remark}
    In fact, by choosing $\beta_i \coloneqq \frac{\gamma}{i \log i (\log \log i)^2}$, we can take the previous $v$ not in $L^\eps$ for any $\eps > 0$.
\end{remark}

\section{Construction in the multi-dimensional case}
\label{sec:multid}

To solve \cref{prob:1} and \cref{prob:1bis} in multiple space-dimensions, we rely on the approach to Monge's optimal transport problem proposed by Sudakov in \cite{MR530375}. It consists of writing (through a disintegration) $\mu_0$ and $\mu_1$ as the superposition of measures concentrated on lower-dimensional sets (typically, 1D segments); solving the lower-dimensional transport problems; and, finally, ``gluing'' all the partial transport maps into a single transport map. After a technical gap was found in Sudakov's paper, his program was still carried out successfully: by Ambrosio and Pratelli in \cite{MR2011032,zbMATH01984440} and by Trudinger and Wang in \cite{MR1854255} for the
Euclidean distance; by Ambrosio, Kirchheim, and Pratelli in \cite{MR2096672} for crystalline norms; by Caffarelli, Feldman, and McCann in \cite{MR1862796} for distances induced by
norms that satisfy certain smoothness and convexity assumptions; by Caravenna for general strictly convex norms \cite{zbMATH05909123}; and by Bianchini and Daneri \cite{zbMATH07000064} for general convex norms on finite-dimensional normed spaces.

We consider Sudakov's optimal transport map (associated, \eg, with a strictly convex norm cost) and its decomposition along one-dimensional transport rays\footnote{~If the cost is given by a norm that is not strictly convex, then these {transport rays} need not be one-dimensional.}. We can then directly apply \cref{th:main-1d} to realize these one-dimensional monotone transport maps as the time-$1$ map of the flows associated with suitable (one-dimensional) vector fields. The last step is to piece them together to define a flow in $\R^d$. 

\begin{theorem}[Exact controllability, $d \ge 1$]\label{th:main-d}
  Let us consider $\mu_0, \mu_1 \in \mathcal{P}_{\mathrm{a.c.}}(\R^d)$ (with continuous densities $\bar \mu_0$ and $\bar \mu_1$, respectively) satisfying conditions \crefrange{it:m1}{it:m3}.
Then, there exists a vector field $v: \R^d \to \R^d$ such that 
\[
\T (x) = \phi(1,x), \qquad x \in \R^d,
\]
where $\T$ is Sudakov's transport map (see \cref{th:sudakov} below) and $\phi$ is a solution of \cref{eq:ode} for $x \in {\supp \mu_0}$.
\end{theorem}

The vector field $v$ constructed in \cref{th:main-d} is defined on the Borel partition of $\R^d$ into optimal transport rays $\left\{I_{\alpha}^1\right\}_{\alpha}$ (\cf \cref{th:sudakov}): \ie, $v(x) \coloneqq v_{\alpha}(x)$ for $\mu_0\Restr{I_\alpha}$-a.e. $x \in I_{\alpha}^1$, and $v_{\alpha}$ is the one-dimensional velocity field obtained in  \cref{th:main-1d} corresponding to the one-dimensional monotone optimal transport map $\T _{\alpha}$ on the oriented line associated with $I_\alpha^1$. Moreover, the flow $\phi$ is the unique flow (almost everywhere) whose trajectories are concentrated on transport rays.

We observe that $v$ constructed our way do not need to be Lipschitz continuous (or even continuous) or satisfy the assumptions of the theory developed by DiPerna--Lions--Ambrosio (see~\cite{MR2096794,MR1022305}); see also \cref{lm:ct}.

\subsection{Preliminaries on Sudakov's theorem} 
\label{ssec:sudakov}

For completeness, following the notation of \cite[Chapter 18]{MR4659653}, let us outline Sudakov's result. We consider Monge's problem with Euclidean norm cost:
\begin{align}\label{eq:monge-d}
\mathrm{M}^d(\mu_0, \mu_1) \coloneqq \min \left\{\int_{\mathbb{R}^d}\|\T(x)-x\| \, \mathrm d \mu_0(x):\, \T: \R^d \to \R^d \text{ and } \mu_1 = \T_{\#} \mu_0\right\}\, .
\end{align}

The first step in Sudakov's approach consists in finding a suitable partition of  $\mathbb{R}^d$ on which the transport occurs (namely such that Kantorovich's optimal plans move the initial mass inside the elements of the partition). Given a \emph{Kantorovich potential}\footnote{~Let us recall the duality formula Monge's problem. We have the equality
\[
\mathrm{M}(\mu_0,\mu_1) \equiv \sup \left\{\int_{\R} f(y) \, \mathrm{d} \mu_1(y)-\int_{\R} f(x) \, \mathrm{d} \mu_0(x): \, f \in \operatorname{Lip}_1\left(\mathbb{R}^d\right)\right\},
\]
where $\operatorname{Lip}_1\left(\mathbb{R}^d\right)$ denotes the set of $1$-Lipschitz functions $\mathbb{R}^d \rightarrow \mathbb{R}$. The optimal functions $f$ are called \emph{Kantorovich potentials}.} $f$ from $\mu_0$ to $\mu_1$ (see~\cite[Theorem 3.17]{MR4659653}), we define 
\[
G(f) \coloneqq \{(a,b) \in \R^d \times \R^d:\, f(b) - f(a) = \|a-b\|\}, 
\]
and consider open oriented segments $I_{\alpha}^1\coloneqq ] a_\alpha, b_\alpha[ \subset \mathbb{R}^d$ (where $\alpha \in \mathfrak A$ is a continuous parameter) whose extreme points belong to $G(f)$ and which are maximal with respect to set-inclusion: these are called \emph{optimal rays}. By definition of Kantorovich potential, all transport has to occur on these rays. We will use the notation  $\R_\alpha^1$ for the oriented line corresponding to $I^1_\alpha$. We call \emph{transport set} (relative to $G$), the union of all transport rays: $\bigcup_{\alpha} I^1_\alpha$. The optimal rays $\left\{I_{\alpha}^1\right\}_{\alpha}$ form a Borel partition  of ${{\rm Conv}(\supp \mu_0 \cup \supp \mu_1)}$ into one-dimensional open segments, up to the sets of their initial points and final points (which are $\Leb^d$-negligible and then also $\mu_0$-negligible).

The second step is decomposing the transport problem and  reducing it to a family of independent one-dimensional transport problems. If $\mu_0^\alpha \coloneqq \mu_0\Restr{I^1_\alpha}$ has no atoms, then the unique transference plan concentrated on a monotone graph in $I_{\alpha}^1 \times I_{\alpha}^1$ is actually concentrated on an optimal transport map $\T _{\alpha}$ (\cf \cref{th:ot1d}). Then the transport map $\T $ for the multi-dimensional problem is obtained by assembling the family $\{\T _{\alpha}\}_{\alpha}$ of one-dimensional maps\footnote{\label{fn:regularity}~The regularity of $\T $ is a delicate issue (see, \eg, the discussion in \cite{zbMATH06379826,MR3277433,zbMATH06938396}). For $d=2$ and using the cost of \cite{zbMATH01984440}, the continuity of the Sudakov's ray-monotone optimal transport map $\T $ on the interior of $\supp \mu_0$ was established in \cite{zbMATH02228669} assuming the following conditions hold: 
    \begin{enumerate}[label=\textbf{R-\arabic*}] 
    \item $\mu_0, \mu_1 \in \mathcal P_{\mathrm{a.c.}}(\R^2)$ with densities $\bar \mu_0$ and $\bar \mu_1$;
    \item $\supp \mu_0$ and $\supp \mu_1$  are compact, convex, and disjoint subsets of $\R^2$;
    \item $\bar \mu_0$ and $\bar \mu_1$ are continuous functions on their respective supports; 
    \item  $\bar \mu_0$ and $\bar \mu_1$ are strictly positive in the interior of their respective supports. 
\end{enumerate}
A more recent refinement is contained in \cite{zbMATH07005087}, removing the strict separation assumption, and using a different geometric set of hypotheses.}. The main technical difficulty in Sudakov's approach (and the flawed point in Sudakov's original contribution \cite{MR530375}, which was subsequently amended in the references mentioned above) is proving that the disintegration of $\Leb^d$ on the optimal rays has indeed non-atomic conditional measures. 

We recall the final statement below, following \cite[Theorems 18.1 and 18.7]{MR4659653}.

\begin{theorem}[Sudakov's optimal transport map]\label{th:sudakov}
 If $\mu_0, \mu_1 \in \mathcal{P}\left(\mathbb{R}^d\right)$ and $\mu_0 \ll \mathscr{L}^d$ is absolutely continuous, then there exists a (Borel measurable) transport map $\T$ from $\mu_0$ to $\mu_1$ satisfying \cref{eq:monge-d}. Furthermore, $\T$ is obtained as follows: $\T(x) \coloneqq  \T_\alpha(x)$ if  $x \in  I^1_\alpha$, where $\T_\alpha:I^1_\alpha \to  I^1_\alpha$ is the monotone transport map on the ray $I^1_\alpha$.  Moreover, if we also have $\mu_1 \ll \mathscr{L}^d$, then we can find an optimal transport map $\T$ for $\mathrm{M}^d(\mu_0, \mu_1)$ and an optimal transport map $\mathrm S$ for $\mathrm{M}^d(\mu_1, \mu_0)$ such that $\mathrm S \circ \mathrm T= \mathrm{Id}$, $\mu_0$-a.e. on $\mathbb{R}^d$ and $\mathrm \T \circ \mathrm S=\mathrm{Id}$, $\mu_1$-a.e. on $\mathbb{R}^d$. 
\end{theorem}

Such a \emph{ray-monotone} map is unique: that is, there exists a unique transport map $\T$ between $\mu_0$ and $\mu_1$ such that, for each maximal transport ray $I^1_\alpha$, $\mathrm T$ is non-decreasing from the segment $I^1_\alpha \cap {\supp \mu_0}$ to the segment $I_\alpha^1 \cap {\supp \mu_1}$.

\subsection{Proof of \texorpdfstring{\cref{th:main-d} (and \cref{th:C})}{Theorem 3.1 (and Theorem C)}}
\label{ssec:proof-multid}

In this section, we prove  \cref{th:main-d}.  If we are able to show that $\mu_0\Restr{I_{\alpha}^1}$ and $\mu_1\Restr{I_{\alpha}^1}$ satisfy the assumptions of \cref{th:main-1d}, we can then conclude by invoking Sudakov's decomposition. 

\begin{proof}[Proof of \cref{th:main-d} (and, therefore, \cref{th:C})]
By \cref{th:sudakov}, there exists a (Borel measurable) transport map $\T$ from $\mu_0$ to $\mu_1$, satisfying \cref{eq:monge-d}, such that $\T(x) = \T_\alpha(x)$ for $\mu_{0}^\alpha$-a.e.  $x \in  I_\alpha^1 \cap {\supp \mu_0}$, where $\T_\alpha: I_\alpha^1 \cap {\supp \mu_0}  \to I_\alpha^1 \cap {\supp \mu_1}$ is the monotone transport map on the line $\R_\alpha^1$, and we recall that the measures $\mu_0^\alpha$ and $\mu_1^\alpha$ are the restrictions of $\mu_0$ and $\mu_1$ to $I_\alpha^1$ respectively (which are well defined pointwise, because $\mu_0$ and $\mu_1$ have continuous densities)\footnote{That is, $\mu_i^\alpha$ form a  disintegration of $\mu_i$ (for $i \in \{0,1\}$) with respect to the partition $\left\{I^1_\alpha\right\}_\alpha$, i.e., 
$
\mu_i=\int \mu_i^\alpha f_{\#} \mu_i(\mathrm{d} \alpha)
$
where $f$ is the \textit{partition function}, $f^{-1}(\alpha)=I^1_\alpha$ (see \cite[Section 452]{MR1857292}).}.  

If the measures $\mu_0^{\alpha}$ and $\mu_1^{\alpha}$ satisfy assumptions of \cref{th:main-1d}, then  $\T _{\alpha}$ is induced by a vector field $v_{\alpha}: I_\alpha^1 \cap {\supp \mu_0} \to \R$ and we conclude the proof.

In other words, it suffices to show that the following conditions hold: 
\begin{enumerate}[label=\textbf{M-\arabic*}$_\alpha$,start=0]
\item \label{it:m0a} $\mu_0^\alpha, \mu_1^\alpha \in \mathcal P_{\mathrm{a.c.}}(I_\alpha^1)$ and their densities, $\bar \mu^{\alpha}_0$ and $\bar \mu^{\alpha}_1$, are continuous functions (in their respective supports); 
\item \label{it:m1a} $\supp \mu_0^{\alpha}$ and $\supp \mu_1^{\alpha}$ are convex; 
        \item \label{it:m3a} $\bar \mu_0^{\alpha} >0$ in $\supp \mu_0^\alpha$ and $\bar \mu_1^{\alpha} >0$ in $\supp \mu_1^\alpha$. 
    \end{enumerate} 
    The validity of the first part of \cref{it:m0a} is a key contribution in Sudakov's disintegration result; in our setting, though, assuming additionally that the densities $\bar \mu_0$ and $\bar \mu_1$ are continuous functions, it is straightforward, and so is the second part of \cref{it:m0a};     \cref{it:m1a} follows from the fact that $\supp \mu_0$ and $\supp \mu_1$ are convex, which yields that their intersection with each transport ray is an interval; and  \cref{it:m3a} holds by definition of restriction. 
\end{proof}
\begin{remark}[On the uniqueness of the flow]\label{rmk:uniqueness}
    We remark that the flow constructed in \cref{th:main-d} is the  unique flow (almost everywhere) whose trajectories are concentrated on a given partition of transport rays. A priori, however, such a flow might not be unique in general, since the regularity of the vector field $v$ is not enough to apply classical results on uniqueness of solutions.

Nevertheless, the constructed flow is the unique solution of the explicit Euler scheme as the step size goes to zero, which may suffice for most practical applications.
\end{remark}

\section{Examples}
\label{sec:examples}

We conclude by presenting a few examples about the construction of a (one-dimensional) velocity $v: \R \to \R$ addressing  \cref{prob:2} for simple transport maps.

 \begin{example}[Transport map with one ``good'' fixed point]\label{ex:supports1}
     Let $\mu_0\coloneqq \chi_{[1,2]} \Leb^1$ and $\mu_1\coloneqq \frac{1}{3}\chi_{[0,3]} \Leb^1$. The monotone transport map between $\mu_0$ and $\mu_1$ is $\T (x)=3x - 3$. A solution to Abel's and Julia's equations can be given explicitly as follows:
     \[\begin{aligned} F(x) &= c + \frac{\log\left(\left|x-\frac{3}{2}\right|\right)}{\log(3)}, && x \in \R, \text{ for any } c \in \R, \\ 
     v(x) &= x\log(3) -\frac{3}{2} \log(3), && x \in \mathbb R.
     \end{aligned}
     \]
     This yields 
     $\phi(t,x) = -3/2 (-1 + 3^t) + 3^t x $, so that $\phi(1,x) = 3x - 3$ solves \cref{prob:2}.
     
     We observe that the map $\T $ has a fixed point,  $\bar x =3/2$ and $v(3/2) = 0$, while $F$ is not defined there.
 \end{example}

\begin{example}[Gaussian measures]\label{ex:gaussian}
Let us consider $\mu_0\coloneqq\mathcal{N}\left(m_0, \sigma_0^2\right)$ and $\mu_1\coloneqq\mathcal{N}\left(m_1, \sigma_1^2\right)$ be two Gaussian measures\footnote{~We recall that, by definition, ${\mathcal{N}} (m,\sigma^2)$ has density $\frac{1}{\sigma \sqrt{2 \pi}} e^{-\frac{1}{2}\left(\frac{x-m}{\sigma}\right)^2}$.} in $\mathbb{R}$. The monotone transport map between $\mu_0$ to $\mu_1$ is given by 
\[
\T(x)=\frac{\sigma_1}{\sigma_0}x- \frac{\sigma_1}{\sigma_0} m_0+m_1
\]
(here, we take $\sigma_0, \, \sigma_1 >0$). $\T$ coincides with the identity map if $m_0 = m_1$ and $\sigma_0 = \sigma_1$; has no fixed points if $\sigma_0 = \sigma_1$ and $m_0 \neq m_1$; and has one fixed point at $\bar x = \sigma_0 \frac{m_0-m_1}{\sigma_0-\sigma_1}$ if $\sigma_0 \neq \sigma_1$. At $\bar x$, the densities of the two measures measures do not coincide. The first case is trivial (as we can take $v \equiv 0$); in the other two, using  \cref{th:main-1d} (in particular, \cref{lem:S_empty_ub} or \cref{lem:S_1_ub}), we can build a suitable velocity field $v$.  In particular, we note that a solution to Abel's and Julia's equations can be given explicitly as follows:  
     \begin{align*}
     F(x) &= c + \frac{\log\left(\left|x - \frac{  \sigma_1  m_0-\sigma_0 m_1}{ \sigma_1-\sigma_0 }\right|\right)}{\log\left(\frac{\sigma_1}{\sigma_0}\right)},  &&x \in \R, \text{ for any } c \in \R, \\
     v(x) &= x \log\left(\frac{\sigma_1}{\sigma_0}\right) -  \log\left(\frac{\sigma_1}{\sigma_0}\right) \frac{\sigma_1 m_0-\sigma_0 m_1  }{{\sigma_1}-\sigma_0}, && x \in \R. 
     \end{align*}
     See \cref{fig:plot-gaussians} for an illustration.
\end{example}
         \begin{figure}
         \centering
         \includegraphics[scale=0.45]{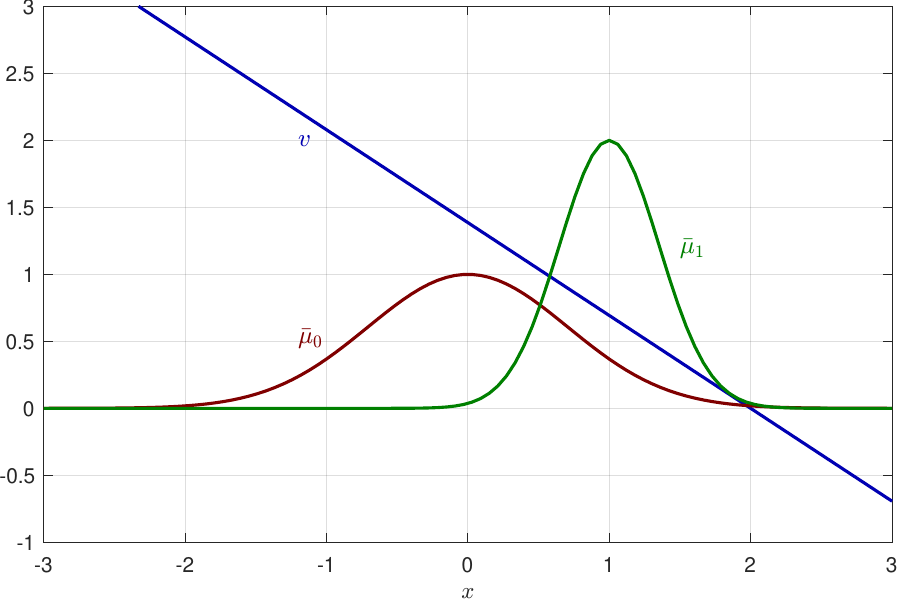}~
         \includegraphics[scale=0.55]{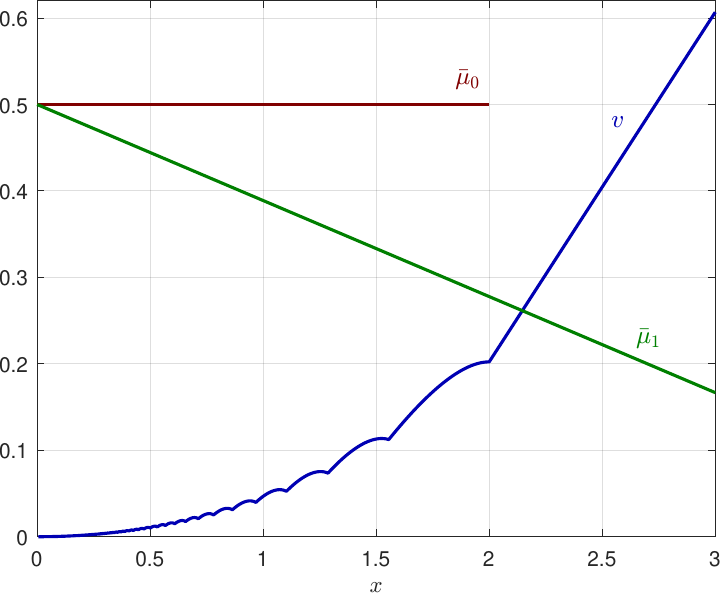}
         \caption{On the left, the vector field transporting a Gaussian $\bar\mu_0 (x) = e^{-x^2}$ into a translated and rescaled Gaussian $\bar \mu_1(x) = 2e^{-4(x-1)^2}$ is given by the linear function $v$ here depicted (as explained in \cref{ex:gaussian} and \cref{ex:affine}). On the right, the densities $\bar\mu_0$ (in red) and $\bar\mu_1$ (in green) from \cref{ex:infty2}.  The  velocity field $v$ (in blue) can be constructed arbitrarily in the interval $[2, 3]$, and this fixes the values uniquely in $[0, 2]$ as well (in this case, we are not trying to match higher derivatives as in \cref{cor:Ck}, so  $v$ is not necessarily $C^1$). Plot created with MATLAB \cite{MATLAB}.
         } 
         \label{fig:plot-gaussians}
     \end{figure}

 \begin{example}[Affine transport maps]\label{ex:affine}
     The previous  examples are particular cases of equivalent measures under affine transformations. Namely, if, in general,
     \[
     \mu_0(\d x) \coloneqq f(x) \mathscr{L}^1(\d x)\quad\text{and}\quad \mu_1(\d x) \coloneqq \alpha f(\alpha(x-\beta))\mathscr{L}^1(\d x)\quad\text{for some}\quad \alpha>0,\,  \beta\in \R,
     \]
     where the density $f$ is positive and continuous in its (convex) support, then the monotone map transporting $\mu_0$ into $\mu_1$ is 
     \[
     \T(x) = \frac{x}{\alpha}+\beta
     \]
     which has a single fixed point at 
     \[
     x_{\alpha\beta} \coloneqq \frac{\alpha\beta}{\alpha-1}.
     \]
     If $\alpha = 1$, this was just a translation and we can fix $v \equiv c$ constant in the whole space. Otherwise, we can take 
     \[
     v(x) = \left\{
     \begin{array}{ll}
     x - x_{\alpha\beta}& \quad\text{if}\quad\alpha\in (0, 1),\\
     x_{\alpha\beta}-x & \quad\text{if}\quad \alpha > 1,
     \end{array}
     \right.
     \]
     and then adjust a multiplicative constant on $v$ so that 
     \[
     \left|\int_0^\beta \frac{\d x}{x-x_{\alpha\beta}} \right|= 1.
     \]
     See Figure~\ref{fig:plot-gaussians} for a sketch in the case of transporting Gaussian measures one into another.
 \end{example}

\begin{example}[Transport map with one ``bad'' fixed point]\label{ex:infty2}
    Let $\mu_0 \coloneqq \frac12 \chi_{[0,2]} \mathscr{L}^1$ and $\mu_1(\mathrm d x) \coloneqq \left( \frac{1}{2}-\frac19 x\right) \chi_{[0,3]}(x) \mathscr{L}^1(\d x)$. The monotone transport map that brings $\mu_1$ to $\mu_0$ is   
    \[
    \T^{-1}(x) = x-\frac19 x^2
    \]
It has a single fixed point at $\bar x = 0$, where the densities of both measures coincide. Thanks to \cref{th:main-1d}, we can construct a velocity field in $[0,3]$, which follows for any arbitrary $v$ fixed in $[2, 3]$. See \cref{fig:plot-gaussians} for one such example.
\end{example}

\begin{example}[Transport map with a sequence of ``good'' fixed points]\label{ex:infty-2}
Let $\mu_0 \coloneqq  \chi_{[0,1]} \mathscr L^1$ and
$\T(x) \coloneqq x + \frac15 x^3\sin(\pi/x)\in C^1([0, +\infty))$, which has fixed points \[\mathcal S = \{0\} \cup \left\{\, \frac{1}{n}: \ n \in \mathbb Z\setminus{\{0\}}   \right\}.\] 
In $\mathcal S$, $0$ is an accumulation point. 
We define  
\(\mu_1 \coloneqq \T_\# \mu_0 \) (so we have $\mu_1 = \bar \mu_1 \mathscr L^1$, with $\bar \mu_1 = (\T^{-1})'\,\chi_{[0,1]}\in C([0, 1])\cap C^\infty((0, 1))$; see \cref{fig:plot-mu1} for an illustration). Moreover,  $\bar \mu_0 \neq \bar \mu_1$ in $\mathcal S\setminus\{0\}$. Then $\mu_0$ and $\mu_1$  satisfy the hypotheses of \cref{th:main-1d}, and we can construct a Lipschitz continuous velocity field solving \cref{prob:2} in $(0, 1]$. 

     \begin{figure}
         \centering
         \includegraphics[scale=0.55]{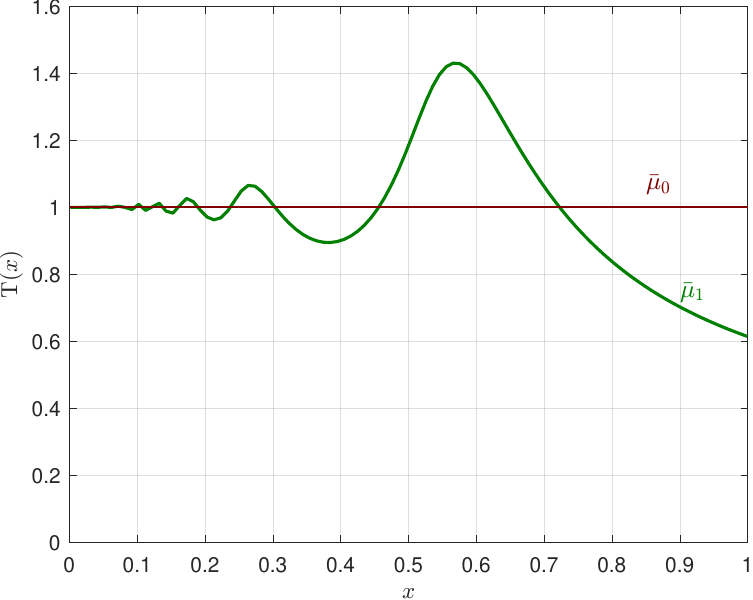}
                \hspace{0.5cm}  \includegraphics[scale=0.55]{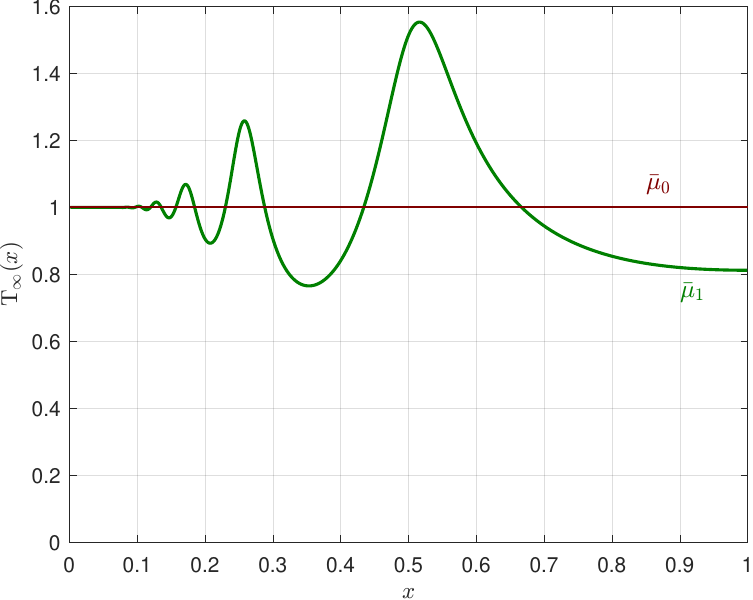}
         \caption{The densities $\bar \mu_0$ (in red) and $\bar \mu_1$ (in green) from \cref{ex:infty-2} for the $C^1$ map $\T$ (left) and for the $C^\infty$ map $\T_\infty$ (right). Plot created with MATLAB \cite{MATLAB}.}
         \label{fig:plot-mu1}
     \end{figure}

In fact, we can even make the previous transport map, and thus $\mu_1$, to be smooth up to the endpoints (\ie, $C^\infty([0, 1])$), by taking, for example, $\T_\infty(x) \coloneqq x + \frac15 e^{-\frac{1}{x}}\sin(\pi/x)\in C^\infty([0, +\infty))$.
\end{example}

\vspace{0.5cm}
\section*{Acknowledgments}

We are grateful to E.~Zuazua for suggesting the problem and for his insightful comments on the manuscript. We also thank M.~Colombo, A.~Figalli, F.~Glaudo, and M.~C.~Zdun for helpful discussions on the topics of this work. 

N.~De Nitti is a member of the Gruppo Nazionale per l’Analisi Matematica, la Probabilità e le loro Applicazioni (GNAMPA) of the Istituto Nazionale di Alta Matematica (INdAM). He has been funded by the Swiss State Secretariat for Education, Research and Innovation (SERI) under contract number MB22.00034 through the project TENSE.

X.~Fern\'andez-Real has been funded by the Swiss National Science Foundation (SNF grant PZ00P2\_208930), by the Swiss State Secretariat for Education, Research and Innovation (SERI) under contract number MB22.00034 through the project TENSE, and by the AEI project PID2021-125021NA-I00 (Spain).

The authors are grateful to the anonymous referees for their very useful corrections and suggestions. 

\vspace{0.5cm}

\printbibliography
\vfill 
\end{document}